\documentclass[a4paper,11pt,fullpage]{article}
\usepackage[english]{babel}

\usepackage[left=1in,right=1in,top=1in,bottom=1in]{geometry}

\usepackage{amsmath}
\usepackage{amsthm}
\usepackage{amsfonts}
\usepackage{amssymb}
\usepackage{amsxtra}
\usepackage{latexsym}
\usepackage{ifthen}
\usepackage{cite}
\usepackage{esint}
\usepackage[cp1250]{inputenc}

\usepackage{epic}
\usepackage{eepic}

\DeclareMathOperator*{\esssup}{ess\,sup}

\DeclareMathOperator*{\essinf}{ess\,inf}

\DeclareMathOperator*{\osc}{osc}

\numberwithin{equation}{section}
\newtheorem{theorem}{Theorem}[section]
\newtheorem{lemma}{Lemma}[section]
\newtheorem{remark}{Remark}[section]
\newtheorem{definition}{Definition}[section]

\newtheorem{example}{Example}[section]

\def\XXint#1#2#3{{\setbox0=\hbox{$#1{#2#3}{\int}$}
     \vcenter{\hbox{$#2#3$}}\kern-.5\wd0}}

\begin{document}

\title{Interior continuity, continuity up to the boundary and Harnack's inequality
for double-phase elliptic equations with non-logarithmic conditions
}

\author{Oleksandr V. Hadzhy, \ Igor I. Skrypnik, \ Mykhailo V. Voitovych
 }

  \maketitle

  \begin{abstract}
We prove continuity and Harnack's inequality for bounded solutions to
elliptic equations of the type
$$
\begin{aligned}
{\rm div}\big(|\nabla u|^{p-2}\,\nabla u+a(x)|\nabla u|^{q-2}\,\nabla u\big)=0,&
\quad a(x)\geqslant0,
\\
|a(x)-a(y)|\leqslant A|x-y|^{\alpha}\mu(|x-y|),&
\quad x\neq y,
\\
{\rm div}\Big(|\nabla u|^{p-2}\,\nabla u \big[1+\ln(1+b(x)\, |\nabla u|) \big] \Big)=0,&
\quad b(x)\geqslant0,
\\
|b(x)-b(y)|\leqslant B|x-y|\,\mu(|x-y|),&
\quad x\neq y,
\end{aligned}
$$
$$
\begin{aligned}
{\rm div}\Big(|\nabla u|^{p-2}\,\nabla u+
c(x)|\nabla u|^{q-2}\,\nabla u \big[1+\ln(1+|\nabla u|) \big]^{\beta} \Big)=0,&
\quad c(x)\geqslant0, \, \beta\geqslant0,\phantom{=0=0}
\\
|c(x)-c(y)|\leqslant C|x-y|^{q-p}\,\mu(|x-y|),&
\quad x\neq y,
\end{aligned}
$$
under the precise choice of $\mu$.



\end{abstract}

\pagestyle{myheadings} \thispagestyle{plain}
\markboth{Igor I. Skrypnik}
{Interior continuity, continuity up to the boundary and Harnack's inequality
for . . . .}

\section{Introduction and main results}\label{Introduction}

Let $\Omega$ be a bounded domain in $\mathbb{R}^{n}$, $n\geqslant2$.
In this paper we are concerned with elliptic equations of the type
\begin{equation}\label{eq1.1}
{\rm div} \mathbf{A}(x, \nabla u)=0, \quad x\in\Omega.
\end{equation}
We suppose that the functions $\mathbf{A}:\Omega\times\mathbb{R}^{n}\rightarrow\mathbb{R}^{n}$
are such that $\mathbf{A}(\cdot,\xi)$ are Lebesgue measurable for all $\xi\in \mathbb{R}^{n}$,
and $\mathbf{A}(x,\cdot)$ are continuous for almost all $x\in\Omega$.
We assume also that the following structure conditions are satisfied
\begin{equation}\label{eq1.2}
\begin{aligned}
\mathbf{A}(x,\xi)\,\xi&\geqslant K_{1}\,g(x,|\xi|)\,|\xi|,
\\
|\mathbf{A}(x,\xi)|&\leqslant K_{2}\,g(x,|\xi|),
\end{aligned}
\end{equation}
where $K_{1}$, $K_{2}$ are positive constants. 

This type of equations belongs to a wide class of elliptic equations with generalized
Orlicz growth. In terms of the function $g$, this class can be characterized as follows.
Let $g(x, {\rm v}):\Omega\times \mathbb{R}_{+}\rightarrow \mathbb{R}_{+}$ be a non-negative
function satisfying the following properties: for any $x\in\Omega$
 the function ${\rm v}\rightarrow g(x,{\rm v})$ is increasing and
 $\lim\limits_{{\rm v}\rightarrow0}g(x,{\rm v})=0$,
 $\lim\limits_{{\rm v}\rightarrow +\infty}g(x,{\rm v})=+\infty$,
 and $c_{0}^{-1}\leqslant g(x,1)\leqslant c_{0}$ with some positive $c_{0}$.
 In addition, the following conditions hold:
 \begin{itemize}
\item[(${\rm g}_{1}$)]
There exist $1<p<q$ such that
for $x\in \Omega$ and for ${\rm w}\geqslant{\rm v}> 0$ there holds
\begin{equation*}\label{gqineq}
\left( \frac{{\rm w}}{{\rm v}} \right)^{p-1} \leqslant\frac{g(x, {\rm w})}{g(x, {\rm v})}\leqslant
 \left( \frac{{\rm w}}{{\rm v}} \right)^{q-1}.
\end{equation*}
\end{itemize}
\begin{itemize}
\item[(${\rm g}_{2}$)]
Fix $R>0$ such that $B_{R}(x_{0})\subset\Omega$.
There exists $c>0$ and positive, continuous and non-decreasing function
$\lambda(r)$ on the interval $(0,R)$, $\lambda(r)\leqslant1$,
$\lim\limits_{r\rightarrow0}r^{1-\delta_{0}}/\lambda(r)=0$, and
$\lambda(r)\leqslant(3/2)^{1-\delta_{0}}\lambda(r/2)$ with some $\delta_{0}\in (0,1)$,
such that for any $K>0$ there holds
$$
g(x,  {\rm v}/r)\leqslant c\,K^{c}\, g(y,  {\rm v}/r),
$$
for any $x, y\in B_{r}(x_{0})\subset B_{R}(x_{0})$
and for all
$r\leqslant {\rm v}\leqslant K\lambda(r)$.
\end{itemize}

We note that condition (${\rm g}_{1}$) and conditions
$({\rm aDec})_{q}$, $({\rm aInc})_{p}$ from
\cite{HarHastLee18} coincide. Moreover, in the case $\lambda(r)=1$
condition (${\rm g}_{2}$) and condition (${\rm A}1$-n) from \cite{HarHastLee18} are equivalent.

Sometimes we will assume that condition (${\rm g}_{2}$) holds for $x_{0}\in\partial\Omega$,
in this case we will assume that there exists $R>0$ such that for every ${\rm v}>0$ the
function $g(\cdot,  {\rm v})$ is defined in $B_{R}(x_{0})$ and for any
$x,y\in B_{r}(x_{0})\subset B_{R}(x_{0})$ condition (${\rm g}_{2}$) is valid.

\begin{remark}
{\rm
The function $g_{a(x)}({\rm v}):={\rm v}^{\,p-1}+a(x){\rm v}^{\,q-1}$, ${\rm v}>0$, where
$a(x)\geqslant0$,
$$
|a(x)-a(y)|\leqslant A|x-y|^{\alpha}\mu(|x-y|), \ \  x,y\in \Omega, \ \
x\neq y,
$$
$$
A>0, \ \ 0<q-p\leqslant\alpha\leqslant1, \ \
\lim\limits_{r\rightarrow0}\mu(r)=+\infty, \ \
\lim\limits_{r\rightarrow0}r^{\alpha}\mu(r)=0,
$$
satisfies condition (${\rm g}_{2}$)
with $\lambda(r)=1$, if $q-p<\alpha\leqslant1$ and $r^{\alpha+p-q}\mu(r)\leqslant1$.
Indeed,
\begin{multline*}
g_{a(x)}\left({\rm v}/r\right)
-g_{a(y)}\left({\rm v}/r\right)
\leqslant
|a(x)-a(y)|\,\left(\frac{{\rm v}}{r}\right)^{q-1}
\\
\leqslant A\,r^{\alpha+p-q}\,\mu(r)\,{\rm v}^{\,q-p}\,
\left(\frac{{\rm v}}{r}\right)^{p-1}
\leqslant AK^{q-p}\,\left(\frac{{\rm v}}{r}\right)^{p-1}
\leqslant AK^{q-p}\,g_{a(y)}\left({\rm v}/r\right),
\ \ \text{if} \ \ r\leqslant{\rm v}\leqslant K.
\end{multline*}

Moreover, the function $g_{a(x)}(\cdot)$ satisfies condition
(${\rm g}_{2}$) with $\lambda(r)=\mu^{-\frac{1}{q-p}}(r)$, if $\alpha=q-p$.
Indeed,
$$
\begin{aligned}
g_{a(x)}\left({\rm v}/r\right)
-g_{a(y)}\left({\rm v}/r\right)
\leqslant A\,\mu(r)\,{\rm v}^{\,q-p}&\left(\frac{{\rm v}}{r}\right)^{p-1}
\\
\leqslant
AK^{q-p}&\left(\frac{{\rm v}}{r}\right)^{p-1}\leqslant
AK^{q-p}g_{a(y)}\left({\rm v}/r\right),
\ \ \text{if} \ r\leqslant {\rm v}\leqslant K\lambda(r).
\end{aligned}
$$

The function
$g_{b(x)}({\rm v}):= {\rm v}^{\,p-1}\Big[ 1+\ln\big(1+b(x){\rm v}\big) \Big]$,
${\rm v}>0$, where $b(x)\geqslant0$,
$$
|b(x)-b(y)|\leqslant B|x-y|\,\mu(|x-y|), \ \  x,y\in \Omega, \ \
x\neq y,
$$
$$
B>0, \ \
\lim\limits_{r\rightarrow0}\mu(r)=+\infty, \ \
\lim\limits_{r\rightarrow0}r\mu(r)=0,
$$
satisfies condition (${\rm g}_{2}$)
with $\lambda(r)=1/\mu(r)$. Indeed,
\begin{multline*}
g_{b(x)}\left(\frac{\lambda(r){\rm v}}{r}\right)
-g_{b(y)}\left(\frac{\lambda(r){\rm v}}{r}\right)
\\
\leqslant
\left(\frac{\lambda(r){\rm v}}{r}\right)^{p-1}
\left|\ln \frac{1+b(x)\dfrac{\lambda(r){\rm v}}{r}}
{1+b(y)\dfrac{\lambda(r){\rm v}}{r}} \right|
\leqslant
\left(\frac{\lambda(r){\rm v}}{r}\right)^{p-1}
\ln\left( 1+|b(x)-b(y)|\, \frac{\lambda(r){\rm v}}{r} \right)
\\
\leqslant
\left(\frac{\lambda(r){\rm v}}{r}\right)^{p-1}
\ln\Big(1+B{\rm v}\lambda(r)\mu(r)\Big)
\leqslant \left(\frac{\lambda(r){\rm v}}{r}\right)^{p-1}
\ln(2+BK), \ \ \text{if} \ \ r\leqslant{\rm v}\leqslant K.
\hskip 7,8mm
\end{multline*}

Similarly, the function
$g_{c(x)}({\rm v}):={\rm v}^{\,p-1}+c(x){\rm v}^{\,q-1}
\left[1+\ln(1+{\rm v})\right]^{\beta}$, ${\rm v}>0$, $\beta\geqslant0$,
$c(x)\geqslant0$,
$$
|c(x)-c(y)|\leqslant C|x-y|^{q-p}\mu(|x-y|),
\ \ x,y\in \Omega, \ \ x\neq y,
$$
$$
C>0, \ \
\lim\limits_{r\rightarrow0}\mu(r)=+\infty, \ \
\lim\limits_{r\rightarrow0}r^{q-p}\mu(r)=0,
$$
satisfies condition (${\rm g}_{2}$) with
$\lambda(r)=\mu^{-\frac{1}{q-p}}(r)\ln^{-\frac{\beta}{q-p}}\dfrac{1}{r}$.
}
\end{remark}
\begin{remark}
{\rm Let's consider the functions
$g_{1}(x,{\rm v})={\rm v}^{\,p(x)-1}$,
$g_{2}(x,{\rm v})={\rm v}^{\,p-1}\big(1+b(x)\ln(1+{\rm v}) \big)$,
${\rm v}>0$, $b(x)\geqslant 0$ in $\Omega$,
$$
|p(x)-p(y)|+|b(x)-b(y)|\leqslant L\left( \ln \frac{1}{r\mu(r)} \right)^{-1},
\quad x,y\in B_{r}(x_{0}),
$$
$$
0<L<+\infty, \quad \lim\limits_{r\rightarrow 0}\mu(r)=+\infty,
\quad \lim\limits_{r\rightarrow 0}r^{1-\delta_{0}}\mu(r)=0 \
\text{ with some } \ \delta_{0}\in(0,1).
$$

It is obvious that the functions $g_{1}$, $g_{2}$ satisfy condition (${\rm g}_{2}$)
with $\lambda(r)=1/\mu(r)$. By our choices the following inequalities hold:
$\delta_{0}\ln\dfrac{1}{r}\leqslant \ln \dfrac{1}{r\mu(r)}\leqslant \ln \dfrac{1}{r}$.
So, in this case condition (${\rm g}_{2}$) is equivalent to the logarithmic Zhikov's condition
\cite{ZhikJMathPh94}.
In this case, the qualitative properties of solutions are well known
(see e.g. \cite{Alhutov97, AlhutovMathSb05, AlhutovKrash04, AlkhSurnApplAn19, AlkhSurnJmathSci20, Fan1995, FanZhao1999, HarHasZAn19, HarHastLee18, HarHastToiv17, Krash2002}).

}
\end{remark}

The aim of this paper is to establish basic qualitative properties such as interior
continuity of bounded solutions, their continuity up to the boundary and Harnack's inequality
for non-negative bounded solutions to Eq. \eqref{eq1.1}.

Before formulating the main results, let us recall the definition of a bounded weak solution
to Eq. \eqref{eq1.1}. We set $G(x,{\rm v}):=g(x,{\rm v}){\rm v}$ for $x\in \Omega$ and
${\rm v}\geqslant0$ and write $W(\Omega)$ for the class of functions
$u\in W^{1,1}(\Omega)$ with $\int\limits_{\Omega}G(x,|\nabla u|)\,dx<+\infty$.
%
%
We also need a class of functions $W_{0}(\Omega)$ which consists of functions
$u\in W_{0}^{1,1}(\Omega)$ 
such that $\int\limits_{\Omega}G(x,|\nabla u|)\,dx<+\infty$.
\begin{definition}
{\rm
We say that a function $u\in W(\Omega)\cap L^{\infty}(\Omega)$ is a bounded weak
sub(super)-solution
to Eq. \eqref{eq1.1} if 
\begin{equation}\label{eq1.5}
\int_{\Omega}\mathbf{A}(x, \nabla u)\,\nabla\varphi \,dx\leqslant(\geqslant)\,0,
\end{equation}
holds for all non-negative test functions $\varphi\in W_{0}(\Omega)$.
}
\end{definition}
\begin{remark}
{\rm
Note that we are dealing only with bounded solutions. To prove the boundedness
we need additional assumptions either on the function $g$ (see, for example
\cite{BarColMingCalc.Var.18, ColMing218, ColMing15, HarHastLee18, HarHastToiv17}),
or on the solution itself (see \cite{BenHarHasKarp20, OkNA20}).
}
\end{remark}

Our first main result of this paper reads as follows.

\begin{theorem}\label{th1.1}
Fix $x_{0}\in \Omega$, let $u$ be a bounded weak solution to Eq. \eqref{eq1.1}
and let the assumptions {\rm (${\rm g}_{1}$)}, {\rm (${\rm g}_{2}$)} be fulfilled.
Then there exists positive number $C_{1}$ depending only on $K_{1}$, $K_{2}$, $n$, $p$, $q$,
$c_{0}$, $c$ and $M:=\esssup\limits_{\Omega}|u|$ such that 
\begin{equation}\label{eq1.7}
\osc\limits_{B_{r}(x_{0})}u\leqslant
2M\exp\left( -C_{1}\int\limits_{2r}^{\rho}\lambda(t)\,\frac{dt}{t} \right)
+C_{1}\,\frac{\rho}{\lambda(\rho)},
\end{equation}
for any $0<2r<\rho<R$. If additionally
\begin{equation}\label{eq1.8}
\int\limits_{0}\lambda(t)\,\frac{dt}{t}=+\infty,
\end{equation}
then $u$ is continuous at $x_{0}$.
\end{theorem}
\begin{remark}
{\rm
For the function $g_{a(x)}(\cdot)$ condition \eqref{eq1.8}
is evidently fulfilled in the case $\alpha>q-p$ and $\mu(r)=\ln^{\beta}\dfrac{1}{r}$,
$\beta>0$, since in this case $\lambda(r)=1$.
In the case $\alpha=q-p$, condition \eqref{eq1.8} can be rewritten as
$$
\int\limits_{0} \,\mu^{-\frac{1}{q-p}}(r)\,\frac{dr}{r}=+\infty.
$$
This condition holds, for example, if
$\mu(r)=\ln^{\beta}\dfrac{1}{r}$ and $0\leqslant\beta\leqslant q-p$.

For the function $g_{b(x)}(\cdot)$ condition \eqref{eq1.8} can be rewritten as
$$
\int\limits_{0}\frac{1}{\mu(r)}\,\frac{dr}{r}=+\infty.
$$
The function $\mu(r)=\ln \dfrac{1}{r}$ satisfies the above condition.

For the function $g_{c(x)}(\cdot)$ condition \eqref{eq1.8} can be rewritten as
$$
\int\limits_{0}
\mu^{-\frac{1}{q-p}}(r)\ln^{-\frac{\beta}{q-p}}\frac{1}{r}\,
\frac{dr}{r}=+\infty.
$$
The function $\mu(r)=\ln^{\beta_{1}}\dfrac{1}{r}$, $\beta_{1}\geqslant0$,
$\beta_{1}+\beta\leqslant q-p$ satisfies the above condition.
}
\end{remark}

To formulate our next result, we need the notion of the capacity.
Fix $x_{0}\in \mathbb{R}^{n}$, let $E\subset B_{r}(x_{0})\subset  B_{\rho}(x_{0})$
and for any $m>0$ set
$$
C(E, B_{8\rho}(x_{0});m):=\inf\limits_{\varphi\in \mathfrak{M}(E)}
\int\limits_{B_{8\rho}(x_{0})}g(x,m|\nabla \varphi|)\,|\nabla \varphi|\,dx,
$$
where the infimum is taken over the set $\mathfrak{M}(E)$ of all functions
$\varphi\in W_{0}(B_{8\rho}(x_{0}))$ with $\varphi\geqslant1$ on $E$. If $m=1$,
this definition leads to the standard definition of $C_{G}(E, B_{8\rho}(x_{0}))$
capacity (see, e.g. \cite{HarHasZAn19}).

Note that if $E=B_{\rho}(x_{0})$ and $m\leqslant \lambda(8\rho)$, then
$C(B_{\rho}(x_{0}), B_{8\rho}(x_{0});m)\asymp \rho^{n-1}g(x_{0},m/\rho)$.
The notation $f\asymp g$ means that there exists a constant $\gamma>0$ such that
$\gamma^{-1}f\leqslant g\leqslant \gamma f$. Let $\varphi\in W_{0}(B_{8\rho}(x_{0}))$
be such that $0\leqslant\varphi\leqslant1$, $\varphi=1$ in $B_{\rho}(x_{0})$ and
$|\nabla \varphi|\leqslant 8/\rho$. Then by {\rm (${\rm g}_{1}$)} and {\rm (${\rm g}_{2}$)} we
obtain
$$
\begin{aligned}
C(B_{\rho}(x_{0}), B_{8\rho}(x_{0});m)\leqslant
&\int\limits_{B_{8\rho}(x_{0})}g(x,m|\nabla \varphi|)\,|\nabla \varphi|\,dx
\\
\leqslant
\frac{8^{q}}{\rho} &\int\limits_{B_{8\rho}(x_{0})}g(x,m/\rho)\,dx
\leqslant c\, 8^{q} g(x_{0},m/\rho)\,\rho^{n-1}.
\end{aligned}
$$
For the opposite inequality, we need the following analogue of the Young inequality:
\begin{equation}\label{Youngineq}
g(x,a)b\leqslant \varepsilon g(x,a)a+ g(x,b/\varepsilon)b,
\quad \varepsilon,\, a,\, b>0, \ x\in\Omega.
\end{equation}
By {\rm (${\rm g}_{1}$)}, {\rm (${\rm g}_{2}$)} and the definition of $1$-capacity we have
$$
\begin{aligned}
C_{1}(B_{\rho}(x_{0}), B_{8\rho}(x_{0}))\,&g(x_{0},m/\rho)
\\
&\leqslant
g(x_{0},m/\rho)\int\limits_{B_{8\rho}(x_{0})}|\nabla \varphi|\,dx
\leqslant
\int\limits_{B_{8\rho}(x_{0})} g(x,m/\rho)\,|\nabla \varphi|\,dx
\\
&\leqslant \frac{\varepsilon m}{\rho}
\int\limits_{B_{8\rho}(x_{0})} g(x,m/\rho)\,dx+
\int\limits_{B_{8\rho}(x_{0})} g(x,|\nabla \varphi|/\varepsilon )\,|\nabla \varphi|\,dx
\\
&\leqslant
\varepsilon c\,m g(x_{0},m/\rho)\rho^{n-1}+
\int\limits_{B_{8\rho}(x_{0})} g(x,|\nabla \varphi|/\varepsilon )\,|\nabla \varphi|\,dx.
\end{aligned}
$$
Choosing $\varepsilon=\varepsilon_{1}/m$, $\varepsilon_{1}\in(0,1)$, from the last inequality
we obtain
$$
C_{1}(B_{\rho}(x_{0}), B_{8\rho}(x_{0}))\,g(x_{0},m/\rho)
\leqslant
c\,\varepsilon_{1}g(x_{0},m/\rho)\rho^{n-1}+\varepsilon_{1}^{1-q}
\int\limits_{B_{8\rho}(x_{0})} g(x,m|\nabla \varphi| )\,|\nabla \varphi|\,dx.
$$
Since $C_{1}(B_{\rho}(x_{0}), B_{8\rho}(x_{0}))\asymp\rho^{n-1}$, choosing $\varepsilon_{1}$
sufficiently small, from this we arrive at the required inequality.

Further we also need the following definition. Let $x_{0}\in \partial\Omega$, we say
that $x_{0}$ is a regular boundary point of the domain $\Omega$ for Eq. \eqref{eq1.1} if for
every bounded weak solution $u\in W(\Omega)$ of Eq. \eqref{eq1.1} satisfying the condition
$u-f\in  W_{0}(\Omega)$, $f\in C(\overline{\Omega})\cap W(\Omega)$, the following equality holds
$\lim\limits_{\Omega \ni x\rightarrow x_{0}}u(x)=f(x_{0})$.

Our next main result of this paper reads as follows:
\begin{theorem}\label{th1.2}
Let $u$ be a bounded weak solution to Eq. \eqref{eq1.1} and let $g(x,\cdot)$ be defined
in $\overline{\Omega}$. Fix $x_{0}\in \partial\Omega$, assume that condition
{\rm (${\rm g}_{1}$)} holds and
let condition {\rm (${\rm g}_{2}$)} be fulfilled in
$\overline{\Omega}$. Assume also that
\begin{equation}\label{eq1.9}
(\mathbf{A}(x, \xi)- \mathbf{A}(x, \eta))(\xi-\eta)>0,
\quad \xi,\eta\in \mathbb{R}^{n}, \quad \xi\neq\eta.
\end{equation}
If additionally
\begin{equation}\label{eq1.10}
\int_{0}\,
g_{x_{0}}^{-1}\left( \frac{C\big(B_{r}(x_{0})\setminus\Omega, B_{8r}(x_{0});\lambda(r)\big)}
{r^{n-1}} \right)dr=+\infty,
\end{equation}
then $x_{0}$ is a regular boundary point of the domain $\Omega$ for Eq. \eqref{eq1.1}.
Here we use the notation $g_{x_{0}}^{-1}({\rm v})$ for the inverse function to the function
$g({x_{0}},{\rm v})$.
\end{theorem}

It seems, that Theorem \ref{th1.2} is new even in the logarithmic case
(i.e. if $\lambda(r)=1$). In the case of the $p(x)$-Laplace equation and $\lambda(r)=1$
Theorem \ref{th1.2} was proved in \cite{AlkhSurnAlgAn19}. In the case $\lambda(r)=1$
and
$C_{G}(B_{r}(x_{0})\setminus\Omega, B_{8r}(x_{0}))
\geqslant\gamma^{-1}C_{G}(B_{8r}(x_{0}))\asymp g(x_{0},1/r)r^{n-1}$,
Theorem \ref{th1.2} was proved in \cite{HarHasZAn19}. Note also that the inequality
$$
C(B_{r}(x_{0})\setminus\Omega, B_{8r}(x_{0});\lambda(r))\geqslant
\gamma^{-1}C(B_{r}(x_{0}), B_{8r}(x_{0});\lambda(r))
\asymp g(x_{0},\lambda(r)/r)r^{n-1}
$$
implies that
$$
g^{-1}_{x_{0}}
\left( \frac{C(B_{r}(x_{0})\setminus\Omega, B_{8r}(x_{0});\lambda(r))}{r^{n-1}} \right)
\geqslant
\gamma^{-1}\frac{\lambda(r)}{r}.
$$
Therefore, in this case, condition
$\int\limits_{0} \lambda(r)r^{-1}\,dr=+\infty$
implies \eqref{eq1.10}.

The following example shows how the condition \eqref{eq1.10} can be rewritten for the
double-phase elliptic equation.
\begin{example}
{\rm
Let $g(x,\cdot)=g_{a(x)}(\cdot)$, consider separately two cases:
$a(x_{0})=0$ and $a(x_{0})>0$. First, if $a(x_{0})=0$, then evidently we have
\begin{multline*}
g_{x_{0}}^{-1}\left( \frac{C\big(B_{r}(x_{0})\setminus\Omega, B_{8r}(x_{0});\lambda(r)\big)}
{r^{n-1}} \right)
\\
=
\left( \frac{C\big(B_{r}(x_{0})\setminus\Omega, B_{8r}(x_{0});\lambda(r)\big)}
{r^{n-1}} \right)^{\frac{1}{p-1}}
\geqslant
\lambda(r) \left( \frac{C_{p}\big(B_{r}(x_{0})\setminus\Omega, B_{8r}(x_{0})\big)}
{r^{n-1}} \right)^{\frac{1}{p-1}}.
\end{multline*}
So, in the case $a(x_{0})=0$, condition
$$
\int\limits_{0}\lambda(r)
\left( \dfrac{C_{p}\big(B_{r}(x_{0})\setminus\Omega, B_{8r}(x_{0})\big)}
{r^{n-p}} \right)^{\frac{1}{p-1}} \dfrac{dr}{r}=+\infty
$$
implies \eqref{eq1.10}.

Now, let $a(x_{0})>0$, choose $\rho$ from the condition
$$
A\,\rho^{\alpha} \mu(\rho)=\frac{1}{2}\,a(x_{0}),
$$
this choice of $\rho$ guarantees that
$\frac{1}{2}\,a(x_{0})\leqslant a(x)\leqslant \frac{3}{2}\,a(x_{0})$ for $x\in B_{\rho}(x_{0})$.
In this case we have
$$
\frac{C\big(B_{r}(x_{0})\setminus\Omega, B_{8r}(x_{0});\lambda(r)\big)}
{r^{n-1}}\geqslant
\frac{1}{2}\,a(x_{0})\,\lambda^{q-1}(r)\,
\frac{C_{q}\big(B_{r}(x_{0})\setminus\Omega, B_{8r}(x_{0})\big)}
{r^{n-1}}.
$$
By (${\rm g}_{1}$) and the Young inequality
we obtain
$$
g_{a(x_{0})}\Bigg(\lambda(r)
\bigg(\frac{C_{q}\big(B_{r}(x_{0})\setminus\Omega, B_{8r}(x_{0})\big)}
{r^{n-1}} \bigg)^{\frac{1}{q-1}}\Bigg)\leqslant
1+\big(1+a(x_{0})\big)  \lambda^{q-1}(r)\,
\frac{C_{q}\big(B_{r}(x_{0})\setminus\Omega, B_{8r}(x_{0})\big)}
{r^{n-1}} ,
$$
and hence
$$
\lambda(r)
\bigg(\frac{C_{q}\big(B_{r}(x_{0})\setminus\Omega, B_{8r}(x_{0})\big)}
{r^{n-1}} \bigg)^{\frac{1}{q-1}}
\leqslant \gamma+\gamma g^{-1}_{x_{0}}
\bigg(\big(1+a(x_{0})\big)  \lambda^{q-1}(r)\,
\frac{C_{q}\big(B_{r}(x_{0})\setminus\Omega, B_{8r}(x_{0})\big)}
{r^{n-1}}\bigg).
$$
Therefore, in the case $a(x_{0})>0$, condition
$$
\int\limits_{0}\lambda(r)
\left( \frac{C_{q}\big(B_{r}(x_{0})\setminus\Omega, B_{8r}(x_{0})\big)}
{r^{n-q}} \right)^{\frac{1}{q-1}}\frac{dr}{r}=+\infty
$$
implies \eqref{eq1.10}.
}
\end{example}

Our next result is the Harnack inequality for positive bounded solutions to Eq.~\eqref{eq1.1}.
To formulate this result, we need another property that characterized equations with
non-logarithmic conditions.
\begin{itemize}
\item[(${\rm g}_{3}$)]
Fix $R>0$ such that $B_{R}(x_{0})\subset \Omega$.
We assume that for any $K>0$ there exists $C(K)>0$ and positive, continuous and non-increasing
function $\theta(r)$ on the interval $(0,R)$, $\theta(r)\geqslant1$,
$\lim\limits_{r\rightarrow0}r^{1-\delta_{0}}\theta(r)=0$, and
$\theta(r/2)\leqslant (3/2)^{1-\delta_{0}}\theta(r)$ with some $\delta_{0}\in (0,1)$,
such that
$$
g(x, {\rm v}/r)\leqslant C(K)\,\theta(r)\,g(y, {\rm v}/r)
$$
for any $x,y\in B_{r}(x_{0})\subset B_{R}(x_{0})$ and for all
$r\leqslant{\rm v}\leqslant K$.
\end{itemize}
\begin{remark}
{\rm
The function $g_{a(x)}(\cdot)$ satisfies condition (${\rm g}_{3}$) with
$$
\theta(r)=\begin{cases}
 1
& \text{if } \ q-p<\alpha\leqslant 1 \ \text{and} \ r^{\alpha+p-q}\mu(r)\leqslant1, \\[0pt]
 \mu(r)
& \text{if } \ \alpha=q-p.
\end{cases}
$$
Indeed, 
\begin{multline*}
g_{a(x)}({\rm v}/r)-g_{a(y)}({\rm v}/r)
\leqslant
Ar^{\alpha+p-q}\mu(r)\, {\rm v}^{\,q-p}\left( \frac{{\rm v}}{r} \right)^{p-1}
\\
\leqslant
AK^{q-p}\,\theta(r)\left( \frac{{\rm v}}{r} \right)^{p-1} \leqslant
AK^{q-p}\,\theta(r)\, g_{a(y)}({\rm v}/r), \ \
\text{if} \ r\leqslant {\rm v}\leqslant K. \hskip 9mm
\end{multline*}

Similarly, the function $g_{b(x)}(\cdot)$ satisfies condition (${\rm g}_{3}$) with
$\theta(r)=\ln\mu(r)$, indeed
$$
\begin{aligned}
g_{b(x)}({\rm v}/r)-g_{b(y)}({\rm v}/r)
&\leqslant \left( \frac{{\rm v}}{r} \right)^{p-1}
\ln\left( 2+|b(x)-b(y)|\,\frac{{\rm v}}{r} \right)
\\
&\leqslant
\left( \frac{{\rm v}}{r} \right)^{p-1}\ln\big(2+BK\mu(r)\big)
\\
&\leqslant \gamma(B,K) \left( \frac{{\rm v}}{r} \right)^{p-1}\ln\mu(r)
\leqslant \gamma(B,K)\ln\mu(r)\, g_{b(y)}({\rm v}/r),
\end{aligned}
$$
if $R$ is small enough.

Obviously, the function $g_{c(x)}(\cdot)$ satisfies condition (${\rm g}_{3}$)
with $\theta(r)=\ln^{\beta}\dfrac{1}{r}$.
}
\end{remark}
\begin{theorem}\label{th1.3}
Let $u$ be a bounded non-negative weak solution to Eq. \eqref{eq1.1}.
Fix $x_{0}\in\Omega$ and assume that conditions
{\rm (${\rm g}_{1}$)}, \eqref{eq1.8}, \eqref{eq1.9} hold,
and let conditions {\rm (${\rm g}_{2}$)}, {\rm (${\rm g}_{3}$)}
be fulfilled in some ball $B_{R}(x_{0})\subset\Omega$.
Assume also that
$$
\int\limits_{0}
\frac{\theta^{-\frac{2n}{p-1}}(r)}{\lambda(r)}\,\frac{dr}{r}=+\infty.
$$
Then there exists positive number $C_{2}$ depending only on
$n$, $p$, $q$, $K_{1}$, $K_{2}$, $c_{0}$, $c$, $C(M)$ such that either
\begin{equation}\label{eq1.11}
u(x_{0})\leqslant C_{2}\, \rho\, \frac{\,\theta^{\frac{2n}{p-1}}(\rho)}{\lambda(\rho)},
\end{equation}
or
\begin{equation}\label{eq1.12}
u(x_{0})\leqslant  C_{2} \, \frac{\theta^{\frac{2n}{p-1}}(\rho)}{\lambda(\rho)}
\inf\limits_{B_{\rho}(x_{0})}u,
\end{equation}
for all $0<\rho<R/8$.
\end{theorem}
\begin{remark}
{\rm
We note that for the case $g(x,\cdot)=g_{a(x)}(\cdot)$ and
$q-p<\alpha\leqslant1$, $\mu(\rho)=\ln^{\beta}\dfrac{1}{\rho}$, $\beta>0$,
$\lambda(\rho)=\theta(\rho)=1$, inequalities \eqref{eq1.11}, \eqref{eq1.12}
can be rewritten as
$$
u(x_{0})\leqslant C_{2}\left(\inf\limits_{B_{\rho}(x_{0})}u+\rho\right).
$$
We also note that in the case $g(x,\cdot)=g_{a(x)}(\cdot)$ and
$\alpha=q-p$,
$\mu(\rho)=\ln^{\beta}\dfrac{1}{\rho}$,
$\lambda(\rho)=\ln^{-\frac{\beta}{q-p}}\dfrac{1}{\rho}$,
$\theta(\rho)=\ln^{\beta}\dfrac{1}{\rho}$,
$0\leqslant\beta\leqslant \dfrac{q-p}{1+\frac{2n}{p-1}(q-p)}$,
inequalities \eqref{eq1.11}, \eqref{eq1.12}
translates into
$$
u(x_{0})\leqslant  C_{2}
\left( \ln\frac{1}{\rho}\, \inf\limits_{B_{\rho}(x_{0})}u
+\rho\ln\frac{1}{\rho} \right).
$$

In addition, note that in the case $g(x,\cdot)=g_{b(x)}(\cdot)$ and
$\mu(\rho)=\ln^{\beta}\dfrac{1}{\rho}$, $\lambda(\rho)=\ln^{-\beta}\dfrac{1}{\rho}$,
$\beta\in(0,1)$, $\theta(\rho)=\beta\ln\ln\dfrac{1}{\rho}$,
inequalities \eqref{eq1.11} and \eqref{eq1.12} can be rewritten as
$$
u(x_{0})\leqslant  C_{2}
\left( \ln\frac{1}{\rho}\, \inf\limits_{B_{\rho}(x_{0})}u
+\rho\ln\frac{1}{\rho} \right).
$$
We also note that similar inequality holds in the case
$g(x,\cdot)=g_{c(x)}(\cdot)$ and $\mu(\rho)=\ln^{\beta_{1}}\dfrac{1}{\rho}$,
$\lambda(\rho)=\ln^{-\frac{\beta+\beta_{1}}{q-p}}\dfrac{1}{\rho}$,
$\theta(\rho)=\ln^{\beta}\dfrac{1}{\rho}$,
$0\leqslant\beta_{1}+\beta \left( 1+\dfrac{2n}{p-1}(q-p)\right)
\leqslant q-p$.

These conditions on $\mu(\rho)$ are worse than in Theorems \ref{th1.1} and \ref{th1.2},
but they are much better than condition, which was known earlier (see \eqref{eq1.13} below).
}
\end{remark}

Before describing the method of proof, we say few words
concerning the history of the problem. The study of regularity of minima
of functionals with non-standard growth has been initiated by Zhikov
\cite{ZhikIzv1983, ZhikIzv1986, ZhikJMathPh94, ZhikJMathPh9798, ZhikKozlOlein94},
Marcellini \cite{Marcellini1989, Marcellini1991}, and Lieberman \cite{Lieberman91},
and in the last thirty years, the qualitative theory
of second order elliptic and parabolic equations with so-called log-condition
(i.e. if $\lambda(r)=1$ and $0<\theta(r)\leqslant L<+\infty$) has been actively developed
(see e.g.
\cite{Alhutov97, AlhutovKrash04, BenHarHasKarp20, BurchSkrPotAn, DienHarHastRuzVarEpn,
HarHastOrlicz, HarHasZAn19, HarHastLee18, HarHastToiv17, SkrVoitUMB19, SkrVoitNA20, VoitNA19}
for references).
Equations of this type and systems of such equations arise in various
problems of mathematical physics (see e.g. the monographs
\cite{AntDiazShm2002_monogr, HarHastOrlicz, Ruzicka2000, Weickert} and references therein).

Double-phase elliptic equations under the logarithmic condition were studied by
Colombo, Mingione \cite{ColMing218, ColMing15, ColMingJFnctAn16}
and by Baroni, Colombo, Mingione \cite{BarColMing, BarColMingStPt16, BarColMingCalc.Var.18}.
Particularly,
$C^{0,\beta}_{{\rm loc}}(\Omega)$, $C^{1,\beta}_{{\rm loc}}(\Omega)$ and Harnack's
inequality were obtained in the case $\lambda(r)=1$ and $\theta(r)=1$ under the precise
conditions on the parameters $\alpha$, $p$, $q$.

The case when conditions (${\rm g}_{2}$) or (${\rm g}_{3}$) hold differs substantionally from the
logarithmic case. To our knowledge there are few results in this direction.
Zhikov \cite{ZhikPOMI04}
obtained a generalization of the logarithmic condition which guarantees the density
of smooth functions in Sobolev space $W^{1,p(x)}(\Omega)$.  Particularly,
this result holds if $1<p\leqslant p(x)$ and
$$
|p(x)-p(y)|\leqslant L\,
\frac{ \Big|\ln \big|\ln |x-y|\big| \Big|}{\big|\ln |x-y|\big|},
\quad x,y\in\Omega, \quad x\neq y, \quad 0<L<p/n.
$$
Later Zhikov and Pastukhova \cite{ZhikPast2008MatSb}
under the same condition proved higher integrability
of the gradient of solutions to the $p(x)$-Laplace equation.
Interior continuity, continuity up to the boundary and Harnack's inequality
to the $p(x)$-Laplace equation were proved by Alkhutov, Krasheninnikova \cite{AlhutovKrash08},
Alkhutov, Surnachev \cite{AlkhSurnAlgAn19} and Surnachev \cite{SurnPrepr2018}
under the condition (${\rm g}_{3}$), where $\theta(r)$ satisfies
\begin{equation}\label{eq1.13}
\int\limits_{0} \exp\big( -\gamma\,\theta^{c}(r) \big)\,\frac{dr}{r}=+\infty,
\end{equation}
with some positive constants $\gamma$, $c>1$. Particularly, the function
$\theta(r)=\left(\ln\ln \dfrac{1}{r}\right)^{L}$, $L<\dfrac{1}{c}$ satisfies the above
condition. These results were generalized in \cite{SkrVoitNA20, ShSkrVoit20}
for a wide class of elliptic and parabolic equations with non-logarithmic Orlicz growth.
Particularly, it was proved in \cite{SkrVoitNA20} that under conditions (${\rm g}_{1}$),
(${\rm g}_{3}$) and \eqref{eq1.13} functions from the corresponding
De\,Giorgi's $\mathcal{B}_{1}(\Omega)$ classes are continuous and moreover, it was shown
that the solutions of the correspondent elliptic and parabolic equations with non-standard
growth belong to these classes.

In the present paper, we substantionally refine the results of \cite{SkrVoitNA20, ShSkrVoit20}.
We would like to mention the approach taken in this paper. To prove interior continuity we use
De\,Giorgi's approach. We consider De\,Giorgi's $\mathcal{B}_{1}(\Omega)$ classes and we prove
the continuity of the functions belonging to these classes. The main difficulty arising in the
proof of the main results is related to the so-called theorem on the expansion of positivity.
Roughly speaking, having information on the measure of the "positivity set" of $u$ over the ball
$B_{r}(\overline{x})$:
$$
\left| \left\{ x\in B_{r}(\overline{x}):
u(x)\leqslant N \right\} \right| \leqslant
\big(1-\beta(r)\big)\, |B_{r}(\overline{x})|
$$
with some $r>0$, $N>0$ and $\beta:=\beta(r)\in(0,1)$, and using the standard De\,Giorgi's or
Moser's arguments, we inevitably come to the estimate
$$
u(x)\geqslant \gamma^{-1}e^{-\gamma \beta^{-\overline{c}}(r)}N, \ \ x\in B_{2r}(\overline{x}),
$$
with some $\gamma$, $\overline{c}>1$. This estimate leads us to condition \eqref{eq1.13}
(see, e.g. \cite{ShSkrVoit20, SkrVoitNA20}). To avoid this, we use a workaround that goes back
to Landis's papers \cite{Landis_uspehi1963, Landis_mngrph71} and his so-called "growth" lemma.
So, we use the following auxiliary solutions. Fix $x_{0}\in\mathbb{R}^{n}$ and let
$E\subset B_{r}(x_{0})\subset B_{\rho}(x_{0})$ and consider the solution
$v:=v(x,m)$ of the following problem:
\begin{equation}\label{eq1.14}
{\rm div} \mathbf{A}(x, \nabla v)=0, \quad x\in \mathcal{D}:=B_{8\rho}(x_{0})\setminus E,
\end{equation}
\begin{equation}\label{eq1.15}
v-m\psi\in W_{0}(\mathcal{D}),
\end{equation}
where $0< m\leqslant \lambda(8\rho)$ is some fixed number, and $\psi\in W_{0}(B_{8\rho}(x_{0}))$,
$\psi=1$ on $E$.

The main step in the proof of Theorems \ref{th1.2} and \ref{th1.3} is the following theorem.
\begin{theorem}\label{th1.4}
Let $v$ be a solution to \eqref{eq1.14}, \eqref{eq1.15} and let conditions {\rm (${\rm g}_{1}$)},
{\rm (${\rm g}_{2}$)} be fulfilled. Then there exists a
positive constant $C_{3}$ depending only on $n$, $p$, $q$, $K_{1}$, $K_{2}$, $c_{0}$, $c$ and
$M$ such that either
\begin{equation}\label{eq1.14pr1}
C(E, B_{8\rho}(x_{0});m)\leqslant C_{3}\rho^{n-1},
\end{equation}
or
\begin{equation}\label{eq1.16}
C_{3}^{-1} g_{x_{0}}^{-1}\left( \frac{C(E, B_{8\rho}(x_{0});m)}{\rho^{n-1}} \right)
\rho \leqslant v(x) \leqslant C_{3}\,
g_{x_{0}}^{-1}\left( \frac{C(E, B_{8\rho}(x_{0});m)}{\rho^{n-1}} \right)\rho,
\end{equation}
for a.a. $x\in B_{2\rho}(x_{0})\setminus B_{\rho}(x_{0})$.
\end{theorem}

The rest of the paper contains the proof of the above theorems.

In Section \ref{Sec2} we consider elliptic $\mathcal{B}_{1}(\Omega)$ classes and prove
interior continuity of functions from these classes. Section \ref{Sec3} contains the upper
and lower bounds for auxiliary solutions of \eqref{eq1.14}, \eqref{eq1.15}.
In Section \ref{Sec4} -- continuity up to the boundary and Harnack's inequality,
proof of Theorems \ref{th1.2}, \ref{th1.3}. We also note that our proofs do not require
studying the special properties of Orlicz spaces.
Finally, we note that in the proof of the main results we do not distinguish between the cases
of so-called $p$-growth (i.e. if $a(x)=0$ or $b(x)=0$) and $(p,q)$-growth
(i.e. if $a(x)>0$ or $b(x)>0$). Moreover, it seems that even in the case $\mu(r)=1$
Eqs. \eqref{eq1.1} with $g(x,\cdot)=g_{b(x)}(\cdot)$ and $g(x,\cdot)=g_{c(x)}(\cdot)$
are considered here for the first time.



\section{Elliptic $\mathcal{B}_{1}$ classes and local continuity.
Proof of Theorem \ref{th1.1}}\label{Sec2}

We refer to the parameters $K_{1}$, $K_{2}$, $n$, $p$, $q$, $c$, $c_{0}$ and $M$ as our structural
data, and we write $\gamma$ if it can be quantitatively determined a priory in terms of the above
quantities. The generic constant $\gamma$ may change from line to line.

As it was already mentioned, elliptic $\mathcal{B}_{1}$ classes were practically defined in
\cite{SkrVoitNA20}.
\begin{definition}
{\rm
We say that a measurable function $u:\Omega\rightarrow \mathbb{R}$ belongs to the elliptic class
$\mathcal{B}_{1,g}(\Omega)$ if $u\in W^{1,1}_{{\rm loc}}(\Omega)\cap L^{\infty}(\Omega)$,
$\esssup\limits_{\Omega} |u|\leqslant M$ and there exists numbers $1<p<q$, $c_{1}>0$ such that for
any ball $B_{8r}(x_{0})\subset\Omega$, any $k$, $l\in \mathbb{R}$, $k<l$, $|k|$, $|l|<M$, any
$\varepsilon\in(0,1]$, any $\sigma\in(0,1)$, for any $\zeta\in C_{0}^{\infty}(B_{r}(x_{0}))$,
$0\leqslant\zeta\leqslant1$, $\zeta=1$ in $B_{r(1-\sigma)}(x_{0})$,
$|\nabla \zeta|\leqslant(\sigma r)^{-1}$, the following inequalities hold:
\begin{multline}\label{eq2.1pr1}
\int\limits_{A^{+}_{k,r}\setminus A^{+}_{l,r}}
g\left( x, \frac{M_{+}(k,r)}{r} \right)
|\nabla u|\,\zeta^{\,q}\,dx
\leqslant
c_{1}\, \frac{M_{+}(k,r)}{\varepsilon\, r}\,
\int\limits_{A^{+}_{k,r}\setminus A^{+}_{l,r}}
g\left( x, \frac{M_{+}(k,r)}{r} \right)dx
\\
+  c_{1}\,\sigma^{-q}\varepsilon^{\,p-1}
\int\limits_{A^{+}_{k,r}} g\left( x, \frac{M_{+}(k,r)}{r} \right)
\frac{(u-k)_{+}}{r}\,dx,
\end{multline}
\begin{multline}\label{eq2.2pr1}
\int\limits_{A^{-}_{l,r}\setminus A^{-}_{k,r}}
g\left( x, \frac{M_{-}(l,r)}{r} \right)
|\nabla u|\,\zeta^{\,q}\,dx
\leqslant
c_{1}\, \frac{M_{-}(l,r)}{\varepsilon\, r}\,
\int\limits_{A^{-}_{l,r}\setminus A^{-}_{k,r}}
g\left( x, \frac{M_{-}(l,r)}{r} \right)dx
\\
+  c_{1}\,\sigma^{-q}\varepsilon^{\,p-1}
\int\limits_{A^{-}_{l,r}} g\left( x, \frac{M_{-}(l,r)}{r} \right)
\frac{(u-l)_{-}}{r}\,dx,
\end{multline}
here $(u-k)_{\pm}:=\max\{\pm(u-k), 0\}$,
$A^{\pm}_{k,r}:=B_{r}(x_{0})\cap \{(u-k)_{\pm}>0\}$,
$M_{\pm}(k,r):=\esssup\limits_{B_{r}(x_{0})} (u-k)_{\pm}$.
}
\end{definition}

Our main result of this Section reads as follows:
\begin{theorem}\label{th2.1pr1}
Let $u\in \mathcal{B}_{1,g}(\Omega)$, fix $x_{0}\in \Omega$ such that
$B_{8r}(x_{0})\subset \Omega$ and let conditions {\rm (${\rm g}_{1}$)},
{\rm (${\rm g}_{2}$)} and \eqref{eq1.8} be fulfilled. Then $u$ is continuous at $x_{0}$.
\end{theorem}

We note that the solutions of Eq. \eqref{eq1.1} belong to the corresponding
$\mathcal{B}_{1,g}(\Omega)$ classes. The proof of inequalities
\eqref{eq2.1pr1} and \eqref{eq2.2pr1} is completely similar to the proof
of \cite[Sec.\,2, inequalities (2.1), (2.2)]{SkrVoitNA20}.
Theorem \ref{th2.1pr1} is an immediate  consequence of the following two lemmas.

\begin{lemma}[De\,Giorgi Type Lemma]\label{lem2.1pr1}
Let $u\in \mathcal{B}_{1,g}(\Omega)$,
let $B_{8r}(x_{0})\subset B_{R}(x_{0})\subset \Omega$, and
$$
\mu^{+}_{r}\geqslant \esssup\limits_{B_{r}(x_{0})}u, \quad
\mu^{-}_{r}\leqslant \essinf\limits_{B_{r}(x_{0})}u, \quad
\omega_{r}:=\mu^{+}_{r}-\mu^{-}_{r}
$$
and set $v_{+}:=\mu^{+}_{r}-u$, $v_{-}:=u-\mu^{-}_{r}$.
Fix $\xi\in(0,1)$, then there exists $\nu_{1}\in(0,1)$ depending only on
$n$, $p$, $q$, $c_{2}$ and $M$ such that if
\begin{equation}\label{eq2.3pr1}
\left| \left\{ x\in B_{r}(x_{0}):v_{\pm}(x)\leqslant \xi\,\lambda(r)\,\omega_{r}\right\} \right|
\leqslant\nu_{1}|B_{r}(x_{0})|,
\end{equation}
then either
\begin{equation}\label{eq2.4pr1}
\xi\,\omega_{r}\leqslant \frac{4\,r}{\lambda(r)},
\end{equation}
or
\begin{equation}\label{eq2.5pr1}
v_{\pm}(x)\geqslant\frac{\xi}{4}\,\lambda(r)\,\omega_{r}
\quad \text{for a.a.} \ x\in B_{r/2}(x_{0}).
\end{equation}
\end{lemma}
\begin{proof}
We provide the proof of \eqref{eq2.5pr1} for $v_{+}$, while the proof for $v_{-}$
is completely similar.
For $j=0,1,2,\ldots$ we set $r_{j}:=\dfrac{r}{2}(1+2^{-j})$,
$k_{j}:=\mu^{+}_{r}-\dfrac{\xi\,\lambda(r)\,\omega_{r}}{2}-
\dfrac{\xi\,\lambda(r)\,\omega_{r}}{2^{j+1}}$. We assume that
$M_{+}(k_{\infty},r/2)\geqslant \dfrac{\xi}{4}\,\lambda(r)\,\omega_{r}$,
because in the opposite case, the required \eqref{eq2.5pr1} is evident.
If \eqref{eq2.4pr1} is violated then $M_{+}(k_{\infty},r/2)\geqslant r$.
In addition, since $M_{+}(k_{j},r)\leqslant \xi\,\lambda(r)\,\omega_{r}$
for $j=0,1,2,\ldots$, then condition {\rm (${\rm g}_{2}$)} is applicable and
we obtain that
$$
g\left(x, \frac{M_{+}(k_{j},r_{j})}{r_{j}}  \right)\asymp
2^{j\gamma} g\left(x_{0}, \frac{M_{+}(k_{j},r_{j})}{r}  \right),
\quad x\in B_{r_{j}}(x_{0}).
$$
Therefore inequality \eqref{eq2.1pr1} with $\varepsilon=1$ can be rewritten as
$$
\int\limits_{A^{+}_{k_{j},r_{j+1}}\setminus A^{+}_{k_{j+1},r_{j+1}}}
|\nabla u|\,dx\leqslant \gamma\,2^{j\gamma}
\xi\,\lambda(r)\,\omega_{r} |A^{+}_{k_{j},r_{j}}|.
$$
From this, using the Sobolev embedding theorem completely similar to that of
\cite[Chap.\,2, Lemma\,6.1]{LadUr}, we arrive at the required \eqref{eq2.5pr1},
which completes the proof of the lemma.
\end{proof}
\begin{lemma}[Expansion of Positivity]\label{lem2.2pr1}
Let $u\in \mathcal{B}_{1,g}(\Omega)$,
let $B_{8r}(x_{0})\subset B_{R}(x_{0})\subset \Omega$ and $\xi\in(0,1)$.
Assume that with some $\alpha\in(0,1)$ there holds
\begin{equation}\label{eq2.6pr1}
\left| \left\{ x\in B_{3r/4}(x_{0}):v_{\pm}(x)
\leqslant \xi\,\lambda(r)\,\omega_{r} \right\} \right|
\leqslant(1-\alpha)\, |B_{3r/4}(x_{0})|.
\end{equation}
Then there exists number $C_{\ast}$ depending only on
$n$, $p$, $q$, $c_{2}$, $M$, $\alpha$ and $\xi$ such either
\begin{equation}\label{eq2.7pr1}
\omega_{r}\leqslant  \frac{C_{\ast}r}{\lambda(r)},
\end{equation}
or
\begin{equation}\label{eq2.8pr1}
v_{\pm}(x)\geqslant C_{\ast}^{-1}\lambda(r)\,\omega_{r}
\quad \text{for a.a.} \ x\in B_{r/2}(x_{0}).
\end{equation}
\end{lemma}
\begin{proof}
We provide the proof of \eqref{eq2.8pr1} for $v_{+}$,
while the proof for $v_{-}$ is completely similar.
We set $k_{s}:=\mu_{r}^{+}-\dfrac{\lambda(r)\,\omega_{r}}{2^{s}}$,
$s=1,2, \ldots,s_{\ast}$, where $s_{\ast}$ is large enough to be chosen later.
For $1\leqslant s\leqslant s_{\ast}$ we assume that
$M_{+}(k_{s},r/2)\geqslant \dfrac{\lambda(r)\,\omega_{r}}{2^{s+1}}$,
since otherwise inequality \eqref{eq2.8pr1} is evident.
If \eqref{eq2.7pr1} is violated, then
$M_{+}(k_{s},r/2)\geqslant C_{\ast}r/2^{s_{\ast}}\geqslant r$ if
$C_{\ast}\geqslant 2^{s_{\ast}}$. In addition, since
$M_{+}(k_{s},r)\leqslant 2^{-s}\lambda(r)\,\omega_{r}$, $s=\overline{1,s_{\ast}}$,
then by  {\rm (${\rm g}_{2}$)} we obtain that
$$
g\left( x,\frac{M_{+}(k_{s},r)}{r} \right)\asymp
g\left( x_{0},\frac{M_{+}(k_{s},r)}{r} \right),
\quad x\in B_{r}(x_{0}).
$$
Therefore inequality \eqref{eq2.1pr1} with
$
\varepsilon=
\left( \dfrac{|A^{+}_{k_{s},r}\setminus A^{+}_{k_{s+1},r}|}{|B_{r}(x_{0})|} \right)^{1/p}
$
can be rewritten as
$$
\int\limits_{A^{+}_{k_{s},r}\setminus A^{+}_{k_{s+1},r}}
|\nabla u|\,\zeta^{\,q}\,dx\leqslant \gamma\,\frac{\lambda(r)\,\omega_{r}}{2^{s}}
\left( \dfrac{|A^{+}_{k_{s},r}\setminus A^{+}_{k_{s+1},r}|}
{|B_{r}(x_{0})|} \right)^{\frac{p-1}{p}}|B_{r}(x_{0})|,
$$
where $\zeta\in C_{0}^{\infty}(B_{r}(x_{0}))$, $0\leqslant\zeta\leqslant1$,
$\zeta=1$ in $B_{3r/4}(x_{0})$, $|\nabla\zeta|\leqslant 4/r$.
From this, using \eqref{eq2.6pr1} and De\,Giorgi-Poincar\'{e} inequality
(see, e.g. \cite[Chap.\,2, Lemma\,3.9]{LadUr}), we arrive at the required \eqref{eq2.8pr1},
which completes the proof of the lemma.
\end{proof}

To complete the proof of Theorem \ref{th2.1pr1} we assume that the following two alternative
cases are possible:
$$
\left|\left\{x\in B_{3\rho/4}(x_{0}):
u(x)\geqslant \mu^{+}_{\rho}-\frac{\omega_{\rho}}{2}\right\} \right|
\leqslant \frac{1}{2}\,|B_{3\rho/4}(x_{0})|
$$
or
$$
\left|\left\{x\in B_{3\rho/4}(x_{0}):
u(x)\leqslant \mu^{-}_{\rho}+\frac{\omega_{\rho}}{2}\right\} \right|
\leqslant \frac{1}{2}\,|B_{3\rho/4}(x_{0})|
$$
for $0<\rho<R$. Assume, for example, the first one. Then by Lemma \ref{lem2.2pr1},
using the fact that
$$
\left\{ x\in B_{3\rho/4}(x_{0}):
u(x)\geqslant \mu_{\rho}^{+}-\frac{\lambda(\rho)\,\omega_{\rho}}{2}  \right\}
\subset
\left\{ x\in B_{3\rho/4}(x_{0}):
u(x)\geqslant \mu_{\rho}^{+}-\frac{\omega_{\rho}}{2}  \right\},
$$
we obtain
$$
\omega_{\rho/2}\leqslant \left(1-C_{\ast}^{-1}\lambda(\rho)\right)
\omega_{\rho}+\frac{C_{\ast}\rho}{\lambda(\rho)}.
$$
Iterating this inequality, we have for any $j\geqslant1$
\begin{equation*}
\omega_{\rho_{j}}\leqslant \omega_{\rho}\prod\limits_{i=0}^{j-1}
\left(1-C_{\ast}^{-1}\lambda(\rho_{i})\right)+ C_{\ast}
\sum\limits_{i=0}^{j-1} \frac{\rho_{i}}{\lambda(\rho_{i})}
\leqslant \omega_{\rho}
\exp\left(-\gamma\sum\limits_{i=0}^{j-1}\lambda(\rho_{i}) \right)
+\frac{\gamma\,\rho}{\lambda(\rho)}, \ \ \rho_{j}=2^{-j}\rho.
\\
\end{equation*}

This completes the proof of Theorem \ref{th2.1pr1}.


\section{Upper and lower estimates of auxiliary solutions.
Proof of Theorem \ref{th1.4}}\label{Sec3}

In this Section we prove upper and lower bounds for auxiliary solutions
$v:=v(x,m)$ to problem \eqref{eq1.14}, \eqref{eq1.15}. The existence of the solutions
$v$ follows from the general theory of monotone operators. We will assume that the following
integral identity holds:
\begin{equation}\label{eq3.1}
\int_{\mathcal{D}}\mathbf{A}(x, \nabla v)\,\nabla\varphi \,dx=0
\quad \text{for any } \ \varphi\in W_{0}(\mathcal{D}).
\end{equation}

Testing \eqref{eq3.1} by $\varphi=(v-m)_{+}$ and by $\varphi=v_{-}$ and using condition
\eqref{eq1.9}, we obtain that $0\leqslant v\leqslant m$.

Further we will assume that inequality \eqref{eq1.14pr1} is violated, i.e.
\begin{equation}\label{eq3.2pr1}
C(E, B_{8\rho}(x_{0});m)\geqslant \overline{c}\rho^{n-1},
\end{equation}
where $\overline{c}\geqslant1$ to be chosen later depending only on the data.


\subsection{Upper bound for the function $v$}
We note that in the standard case (i.e. if $p=q$) the upper bound for the function $v$
was proved in \cite{IVSkr1983}
(see also \cite[Chap.\,8, Sec.\,3]{IVSkrMetodsAn1994}, \cite{IVSkrSelWorks}).

For $i, j=0,1,2, \ldots$ set $k_{j}:= k(1-2^{-j})$, $k>0$ to be chosen later,
$\rho_{i,j}:=2^{-i-j-3}\rho$,
$$
M_{i}:=\esssup\limits_{F_{i}}v, \quad
F_{i}:=\left\{ x\in \mathcal{D}: \frac{\rho}{4}(1+2^{-i})\leqslant |x-x_{0}|
\leqslant \frac{\rho}{2}(3-2^{-i}) \right\}.
$$
Fix $\overline{x}\in F_{i}$ and suppose that $(v(\overline{x})-k)_{+}\geqslant\rho$,
then $M_{i,j}(k_{j}):=\esssup\limits_{B_{\rho_{i,j}}(\overline{x})}(v-k_{j})\geqslant\rho_{i,j}$.
And let
$$
\zeta_{i,j}\in C_{0}^{\infty}(B_{\rho_{i,j}}(\overline{x})), \ \
0\leqslant\zeta_{i,j}\leqslant 1, \ \
\zeta_{i,j}=1 \ \text{in} \ B_{\rho_{i,j+1}}(\overline{x}), \ \
|\nabla \zeta_{i,j}|\leqslant 2^{i+j+4}/\rho.
$$
Note that our choice of $\lambda(r)$ guarantees
the inequality
$$
\lambda(R_{1})\leqslant 2\left(\dfrac{R_{1}}{R_{2}}\right)^{1-\delta_{0}}\lambda(R_{2}),
\quad 0<R_{2}<R_{1}<R.
$$
Then
$v(x)\leqslant m\leqslant\lambda(8\rho)\leqslant\gamma\, 2^{(i+j)\gamma}\,\lambda(\rho_{i,j})$,
$x\in B_{\rho_{i,j}}(\overline{x})$.
Therefore condition (${\rm g}_{2}$) is applicable in $B_{\rho_{i,j}}(\overline{x})$ and we have
by (${\rm g}_{2}$) that
$$
g\left(x, \frac{M_{i,j}(k_{j+1})}{\rho_{i,j}}\right)
\asymp 2^{(i+j)\gamma}
g\left(\overline{x}, \frac{M_{i,j}(k_{j+1})}{\rho_{i,j}}\right),
\quad x\in B_{\rho_{i,j}}(\overline{x}).
$$

So, testing \eqref{eq3.1} by $\varphi=(v-k_{j+1})_{+}\,\zeta_{i,j}^{\,q}$ and using the
Young inequality \eqref{Youngineq}, we obtain
$$
\int\limits_{B_{\rho_{i,j}}(\overline{x})}
G(x, |\nabla(v-k_{j+1})_{+}|)\,\zeta_{i,j}^{\,q}\,dx
\leqslant \frac{\gamma\,2^{(i+j)\gamma}}{\rho}\,
g\left(\overline{x}, \frac{M_{i,j}(k_{j+1})}{\rho_{i,j}}\right)
\int\limits_{B_{\rho_{i,j}}(\overline{x})}(v-k_{j+1})_{+}\,dx.
$$
From this, using again inequality \eqref{Youngineq}, we get
$$
\begin{aligned}
&g\left(\overline{x}, \frac{M_{i,j}(k_{j+1})}{\rho_{i,j}}\right)
\int\limits_{B_{\rho_{i,j}}(\overline{x})}
|\nabla(v-k_{j+1})|\,\zeta_{i,j}^{\,q}\,dx
\\
&\leqslant
\gamma\,2^{(i+j)\gamma}\int\limits_{B_{\rho_{i,j}}(\overline{x})}
g\left(x, \frac{M_{i,j}(k_{j+1})}{\rho_{i,j}}\right)
|\nabla(v-k_{j+1})|\,\zeta_{i,j}^{\,q}\,dx
\\
&\leqslant
\frac{\gamma\,2^{(i+j)\gamma}}{\rho}\,M_{i,j}(k_{j+1})
\int\limits_{A^{+}_{\rho_{i,j}, k_{j+1}}}
g\left(x, \frac{M_{i,j}(k_{j+1})}{\rho_{i,j}}\right)dx
\\
&+\gamma\,2^{(i+j)\gamma}
\int\limits_{B_{\rho_{i,j}}(\overline{x})}
G(x, |\nabla(v-k_{j+1})_{+}|)\,\zeta_{i,j}^{\,q}\,dx
\\
&\leqslant
\frac{\gamma\,2^{(i+j)\gamma}}{\rho}\,
g\left(\overline{x}, \frac{M_{i,j}(k_{j+1})}{\rho_{i,j}}\right)
\left( \frac{M_{i,j}(k_{j+1})}{k}+1 \right)
\int\limits_{B_{\rho_{i,j}}(\overline{x})}
(v-k_{j+1})_{+}\,dx,
\end{aligned}
$$
here $A^{+}_{\rho_{i,j}, k_{j+1}}:=
B_{\rho_{i,j}}(\overline{x})\cap \{v>k_{j+1}\}$.
The last inequality can be rewritten as
$$
\int\limits_{B_{\rho_{i,j}}(\overline{x})}
|\nabla(v-k_{j+1})_{+}|\,\zeta_{i,j}^{\,q}\,dx
\leqslant
\frac{\gamma\, 2^{\gamma(i+j)}}{\rho}
\left( \frac{M_{i+1}}{k}+1 \right)
\int\limits_{B_{\rho_{i,j}}(\overline{x})} (v-k_{j+1})_{+}\,dx.
$$
From this, by standard arguments (see, for example, \cite[Chap.~2]{LadUr}),
choosing $k$ from the condition
$$
k=\gamma\,2^{i\gamma}\rho^{-n}\int\limits_{B_{\rho/2^{i+3}}(\overline{x})} v\,dx
+\gamma\,2^{i\gamma}M_{i+1}^{\frac{n}{n+1}}
\Bigg( \rho^{-n}\int\limits_{B_{\rho/2^{i+3}}(\overline{x})} v\,dx \Bigg)^{\frac{1}{n+1}},
$$
and keeping in mind our assumption that $v(\overline{x})\geqslant k+\rho$,
we arrive at
$$
v(\overline{x})\leqslant
\gamma\,2^{i\gamma}\rho^{-n}\int\limits_{B_{\rho/2^{i+3}}(\overline{x})} v\,dx
+\gamma\,2^{i\gamma}M_{i+1}^{\frac{n}{n+1}}
\Bigg( \rho^{-n}\int\limits_{B_{\rho/2^{i+3}}(\overline{x})} v\,dx\Bigg)^{\frac{1}{n+1}}
+\gamma\rho.
$$
Since $\overline{x}\in F_{i}$ is an arbitrary point, using the Young inequality,
from the previous we obtain for any $\varepsilon\in(0,1)$
\begin{equation}\label{eq3.3}
M_{i}\leqslant \varepsilon M_{i+1}+ \frac{\gamma\,2^{i\gamma}}{\varepsilon^{\gamma}\rho^{\,n}}
\int\limits_{F_{i+1}} v\,dx+\gamma\rho,
\quad i=0,1,2, \ldots.
\end{equation}

Let us estimate the second term on the right-hand side of \eqref{eq3.3}.
For this we set $v_{M_{i+1}}:=\min\{v, M_{i+1}\}$,
by {\rm (${\rm g}_{2}$)} and \eqref{Youngineq} we have
\begin{multline*}
\int\limits_{F_{i+1}} v\,dx=\int\limits_{F_{i+1}} v_{M_{i+1}}\,dx
\leqslant \gamma\rho\int\limits_{\mathcal{D}} |\nabla v_{M_{i+1}}|\,dx
\\
= \gamma\rho\int\limits_{\mathcal{D}}
|\nabla v_{M_{i+1}}|\,\frac{g(x,M_{i}/\rho)}{g(x,M_{i}/\rho)}\,dx
\leqslant
\varepsilon_{1} M_{i}\,\rho^{n}+
 \frac{\gamma\varepsilon_{1}^{-\gamma}\rho}{g(x_{0},M_{i}/\rho)}
\int\limits_{\mathcal{D}} G(x,|\nabla v_{M_{i+1}}|)\,dx
\end{multline*}
with arbitrary $\varepsilon_{1}\in(0,1)$.
Collecting the last two inequalities and choosing $\varepsilon_{1}$ from the condition
$\gamma\,2^{i\gamma}\varepsilon^{-\gamma}\varepsilon_{1}=\varepsilon$, we obtain
$$
M_{i}\leqslant 2\varepsilon M_{i+1}+
\frac{\gamma\,2^{i\gamma}\varepsilon^{-\gamma}\rho^{1-n}}{g(x_{0},M_{i}/\rho)}
\int\limits_{\mathcal{D}} G(x,|\nabla v_{M_{i+1}}|)\,dx+\gamma\rho,
$$
which by \eqref{Youngineq} implies
\begin{equation}\label{eq3.4}
\begin{aligned}
M_{i}\,g(x_{0},M_{i}/\rho)&\leqslant
3\varepsilon M_{i+1}\,g(x_{0},M_{i+1}/\rho)
\\
&+
\gamma\,2^{i\gamma}\varepsilon^{-\gamma}\rho^{1-n}\int
\limits_{\mathcal{D}} G(x,|\nabla v_{M_{i+1}}|)\,dx
+\gamma\rho, \quad i=0,1,2,\ldots .
\end{aligned}
\end{equation}

Let $\psi\in \mathfrak{M}(E)$ be such that
$$
\int\limits_{B_{8\rho}(x_{0})}
g(x,m|\nabla \psi|)\,|\nabla \psi|\,dx
\leqslant
C(E,B_{8\rho}(x_{0});m)+\rho^{n}.
$$
Testing identity \eqref{eq3.1} by $\varphi=v-m\psi$,
by the Young inequality \eqref{Youngineq} we obtain
\begin{equation}\label{eq3.5}
\int\limits_{\mathcal{D}} G(x,|\nabla v|)\,dx
\leqslant \gamma m \int\limits_{B_{8\rho}(x_{0})}
g(x,m|\nabla \psi|)\,|\nabla \psi|\,dx\leqslant
\gamma m\left( C(E, B_{8\rho}(x_{0});m)+\rho^{n} \right).
\end{equation}
Testing \eqref{eq3.1} by $\varphi=v_{M_{i+1}}-M_{i+1}v/m$
and using the Young inequality \eqref{Youngineq} and \eqref{eq3.5}, we have
$$
\int\limits_{\mathcal{D}} G(x,|\nabla v_{M_{i+1}}|)\,dx
\leqslant \frac{\gamma M_{i+1}}{m}
\int\limits_{\mathcal{D}} G(x,|\nabla v|)\,dx
\leqslant
\gamma M_{i+1}\big( C(E, B_{8\rho}(x_{0});m)+\rho^{n} \big).
$$
This inequality and \eqref{eq3.4} imply that
$$
\begin{aligned}
M_{i}\,g(x_{0},M_{i}/\rho)&\leqslant
3\varepsilon M_{i+1}\,g(x_{0},M_{i+1}/\rho)
\\
&+
\gamma\,2^{i\gamma}\varepsilon^{-\gamma} M_{i+1}\,\rho
+\gamma\,2^{i\gamma}\varepsilon^{-\gamma} M_{i+1}\,
\frac{C(E, B_{8\rho}(x_{0});m)}{\rho^{n-1}},
\quad i=0,1,2,\ldots ,
\end{aligned}
$$
which yields for any $\varepsilon_{2}\in (0,1)$
$$
\begin{aligned}
g(x_{0},M_{i}/\rho)&\leqslant
\frac{1}{\varepsilon_{2}}\, \frac{M_{i}}{M_{i+1}}\, g(x_{0},M_{i}/\rho)
+\varepsilon_{2}^{p-1} g(x_{0},M_{i+1}/\rho)
\\
&\leqslant
\left(\frac{3\varepsilon}{\varepsilon_{2}}+\varepsilon_{2}^{\,p-1}\right)
g(x_{0},M_{i+1}/\rho)+
\frac{\gamma\,2^{i\gamma}}{(\varepsilon\varepsilon_{2})^{\gamma}}
\left( \frac{C(E, B_{8\rho}(x_{0});m)}{\rho^{n-1}}+\rho \right)
\\
&\leqslant
\left(\frac{3\varepsilon}{\varepsilon_{2}}+\varepsilon_{2}^{\,p-1}\right)
g(x_{0},M_{i+1}/\rho)+
\frac{\gamma\,2^{i\gamma}}{(\varepsilon\varepsilon_{2})^{\gamma}}\,
\frac{C(E, B_{8\rho}(x_{0});m)}{\rho^{n-1}},
\quad i=0,1,2,\ldots,
\end{aligned}
$$
here we also used inequality \eqref{eq3.2pr1}.
Iterating the last inequality and choosing $\varepsilon_{2}$ and $\varepsilon$ small enough,
we arrive at
$$
g(x_{0},M_{0}/\rho)\leqslant
\gamma \, \frac{C(E, B_{8\rho}(x_{0});m)}{\rho^{n-1}},
$$
which proves the upper bound of the function $v$.


\subsection{Lower bound for the function $v$}

The main step in the proof of the lower bound is the following lemma.
\begin{lemma}
There exist positive numbers $\varepsilon$, $\vartheta\in(0,1)$ depending only on the data
such that
\begin{equation}\label{eq3.7}
\left| \left\{x\in K_{\rho/4,\, 2\rho}:
v(x)\leqslant \varepsilon\rho\,
g_{x_{0}}^{-1}\left( \frac{C(E, B_{8\rho}(x_{0});m)}{\rho^{n-1}} \right) \right\} \right|
\leqslant (1-\vartheta)\, |K_{\rho/4,\, 2\rho}|,
\end{equation}
where $K_{\rho_{1},\rho_{2}}:=B_{\rho_{2}}(x_{0})\setminus B_{\rho_{1}}(x_{0})$.
\end{lemma}
\begin{proof}
Let $\zeta_{1}\in C_{0}^{\infty}(B_{\rho}(x_{0}))$, $0\leqslant \zeta_{1}\leqslant1$,
$\zeta_{1}=1$ in $B_{\rho/2}(x_{0})$, $|\nabla \zeta_{1}|\leqslant 2/\rho$. Testing
\eqref{eq3.1} by $\varphi=v-m\,\zeta_{1}^{\,q}$ and using condition {\rm (${\rm g}_{2}$)}
and the Young inequality \eqref{Youngineq},
we obtain for any $0<\varepsilon_{1}\leqslant \lambda(\rho)$
$$
\begin{aligned}
\int\limits_{\mathcal{D}} G(x, |\nabla v|)\,dx
&\leqslant
\frac{\gamma m}{\rho} \int\limits_{K_{\rho/2,\rho}}
g(x, |\nabla v|)\,\zeta_{1}^{\,q-1} \,dx
\\
&\leqslant\frac{\gamma m}{\varepsilon_{1}}\int\limits_{K_{\rho/2,\rho}}
G(x, |\nabla v|)\,dx+\frac{\gamma m}{\rho}\int\limits_{K_{\rho/2,\rho}}
g(x,\varepsilon_{1}/\rho)\,dx
\\
&\leqslant \frac{\gamma m}{\varepsilon_{1}}\int\limits_{K_{\rho/2,\rho}}
G(x, |\nabla v|)\,dx+\gamma m \,g(x_{0}, \varepsilon_{1}/\rho)\,\rho^{n-1}.
\end{aligned}
$$
Let $\zeta_{2}\in C_{0}^{\infty}(K_{\rho/4,\, 2\rho})$, $0\leqslant \zeta_{2}\leqslant1$,
$\zeta_{2}=1$ in $K_{\rho/2,\, \rho}$, $|\nabla \zeta_{2}|\leqslant 4/\rho$.
Testing \eqref{eq3.1} by $\varphi=v\,\zeta_{2}^{\,q}$ and using the Young inequality
\eqref{Youngineq},
we estimate the first term on the right-hand side of the previous inequality as follows:
$$
\int\limits_{K_{\rho/2,\rho}}G(x, |\nabla v|)\,dx\leqslant
\int\limits_{K_{\rho/4,2\rho}}G(x, |\nabla v|)\,\zeta_{2}^{\,q}\,dx\leqslant
\gamma \int\limits_{K_{\rho/4,2\rho}}G(x, v/\rho)\,dx
\leqslant
\gamma \int\limits_{K_{\rho/4,2\rho}}G(x_{0}, v/\rho)\,dx,
$$
here we also used condition {\rm (${\rm g}_{2}$)} and the estimate
$v\leqslant m\leqslant \lambda(8\rho)$.
Combining the last two inequalities and using the definition of capacity,
we obtain
$$
C(E, B_{8\rho}(x_{0});m)\leqslant
\frac{1}{m}\int\limits_{\mathcal{D}}G(x, |\nabla v|)\,dx
\leqslant
\frac{\gamma}{\varepsilon_{1}}\int\limits_{K_{\rho/4,2\rho}}
G(x, v/\rho)\,dx+\gamma\, g(x_{0},\varepsilon_{1}/\rho)\rho^{n-1}.
$$
Choose $\varepsilon_{1}$ from the condition
$$
g(x_{0},\varepsilon_{1}/\rho)=\overline{\varepsilon}_{1}\,
\frac{C(E, B_{8\rho}(x_{0});m)}{\rho^{\,n-1}},
\quad \overline{\varepsilon}_{1}\in (0,1).
$$
By \eqref{eq3.2pr1} $\varepsilon_{1}\geqslant 8\rho$ if
$\overline{c}=\overline{c}(\overline{\varepsilon}_{1})$
is large enough. To apply condition  {\rm (${\rm g}_{2}$)}, we still need to check
the inequality  $\varepsilon_{1}\leqslant \lambda(8\rho)$.
Indeed, let $\psi\in C_{0}^{\infty}(B_{2\rho}(x_{0}))$, $0\leqslant \psi\leqslant1$,
$\psi=1$ in $B_{\rho}(x_{0})$ and $|\nabla \psi|\leqslant 2/\rho$, then by {\rm (${\rm g}_{2}$)}
we have
\begin{multline*}
C(E, B_{8\rho}(x_{0});m)\leqslant C(B_{\rho}(x_{0}), B_{8\rho}(x_{0});m)
\\
\leqslant
\int\limits_{B_{8\rho}(x_{0})}
g(x, m|\nabla \psi|)\,|\nabla \psi|\,dx
\leqslant
\frac{2}{\rho} \int\limits_{B_{2\rho}(x_{0})}g(x, 2m/\rho)\,dx
\leqslant \gamma\,g(x_{0},m/\rho)\,\rho^{n-1},
\end{multline*}
and therefore, if $\overline{\varepsilon}_{1}=1/2\gamma$, then
$\varepsilon_{1}\leqslant \rho\, g^{-1}_{x_{0}}
\big( \overline{\varepsilon}_{1} \gamma\, g(x_{0}, m/\rho) \big)
\leqslant m \leqslant \lambda(8\rho)$.
So, by our choices, from the previous, we obtain
\begin{equation}\label{eq3.8}
C(E, B_{8\rho}(x_{0});m)
\leqslant \frac{\gamma}{\varepsilon_{1}}
\int\limits_{K_{\rho/4, 2\rho}} G(x_{0},v/\rho)\,dx.
\end{equation}
Let us estimate the term on the right-hand side of \eqref{eq3.8}.
For this we decompose $K_{\rho/4, 2\rho}$ as
$K_{\rho/4, 2\rho}=K'_{\rho/4, 2\rho}\cup K''_{\rho/4, 2\rho}$, where
$$
K'_{\rho/4, 2\rho}:=K_{\rho/4, 2\rho}\cap
\left\{ v\leqslant \varepsilon \rho\, g^{-1}_{x_{0}}
\left( \frac{C(E, B_{8\rho}(x_{0});m)}{\rho^{\,n-1}} \right)
 \right\}, \quad
 K''_{\rho/4, 2\rho}:=K_{\rho/4, 2\rho}\setminus K'_{\rho/4, 2\rho},
$$
and $\varepsilon\in (0,1)$ is small enough to be determined  later.
Recall that
$$
\varepsilon_{1}=\rho\,g^{-1}_{x_{0}}
\left( \overline{\varepsilon}_{1} \frac{C(E, B_{8\rho}(x_{0});m)}{\rho^{\,n-1}} \right)
\geqslant \overline{\varepsilon}^{\,1/(p-1)}_{1}\rho\,
g^{-1}_{x_{0}}
\left( \frac{C(E, B_{8\rho}(x_{0});m)}{\rho^{\,n-1}} \right),
$$
so we have
\begin{equation}\label{eq3.9}
\frac{\gamma}{\varepsilon_{1}}\int\limits_{K'_{\rho/4, 2\rho}}
G(x_{0},v/\rho)\,dx \leqslant \varepsilon \gamma\,
C(E, B_{8\rho}(x_{0});m).
\end{equation}
Moreover, by the upper bound for the function $v$, \eqref{eq3.2pr1}
and our choice of $\varepsilon_{1}$, we obtain
\begin{equation}\label{eq3.10}
\frac{\gamma}{\varepsilon_{1}}\int\limits_{K''_{\rho/4, 2\rho}}
G(x_{0},v/\rho)\,dx
\leqslant
\gamma\, \frac{C(E, B_{8\rho}(x_{0});m)}{\rho^{\,n}}\,
|K''_{\rho/4, 2\rho}|.
\end{equation}
Collecting estimates \eqref{eq3.8}--\eqref{eq3.10} and choosing $\varepsilon$ from the condition
$\varepsilon\gamma=1/2$,
 we arrive at
$|K''_{\rho/4,\, 2\rho}|\geqslant \gamma^{-1} |K_{\rho/4,\, 2\rho}|$,
which completes the proof of the lemma.
\end{proof}

Repeating the similar arguments as in Section \ref{Sec2}, by Lemma \ref{lem2.2pr1}
we arrive at \eqref{eq1.16}, which completes the proof of Theorem \ref{th1.4}.


\section{Boundary continuity and Harnack's inequality,
proof of Theorems \ref{th1.2} and \ref{th1.3}}\label{Sec4}

In this Section we will prove Theorems \ref{th1.2} and \ref{th1.3}.

\subsection{Boundary continuity, proof of Theorem \ref{th1.2}}

\subsubsection{Auxiliary propositions}

Further we will need the following two simple lemmas.
\begin{lemma}[Weak Comparison Principle]\label{lem4.1pr1}
Suppose that $u$ is a bounded weak super-solution and $v$
is a bounded weak sub-solution to Eq. \eqref{eq1.1} in $\Omega$.
Assume also that \eqref{eq1.9} holds and $v\leqslant u$ on $\partial\Omega$,
then $v\leqslant u$ a.e. in $\Omega$.
\end{lemma}
\begin{proof}
If not, test the correspondent identities for $v$ and $u$ by $(v-u)_{+}$,
subtracting the resulting inequalities and using condition \eqref{eq1.9},
we arrive to a contradiction. This proves the lemma.
\end{proof}
\begin{lemma}\label{lem4.2pr1}
Let $u$ be a bounded weak solution to Eq. \eqref{eq1.1} in $\Omega$ and
$u-f\in W_{0}(\Omega)$ and assume that \eqref{eq1.9} holds. Fix $x_{0}\in \partial\Omega$
and take any number $k\geqslant \sup\limits_{B_{\rho}(x_{0})\cap\partial\Omega}f$ and
define
$$
u_{k}^{+}:=\begin{cases}
 (u-k)_{+}
& \text{if } \ x\in B_{\rho}(x_{0})\cap\Omega, \\[0pt]
 0
& \text{if } \ x\in B_{\rho}(x_{0})\setminus\Omega.
\end{cases}
$$
Then $u_{k}^{+}$ is a bounded weak sub-solution to Eq. \eqref{eq1.1} in the ball
$B_{\rho}(x_{0})$. The same conclusion holds for the zero extension of
$u_{k}^{-}=(u-k)_{-}$ for the levels
$k\leqslant \inf\limits_{B_{\rho}(x_{0})\cap\partial\Omega}f$.
\end{lemma}
\begin{proof}
Let $\varphi\in W_{0}(B_{\rho}(x_{0}))$, $\varphi\geqslant0$ be arbitrary and test
\eqref{eq1.5} by $\dfrac{u_{k}^{+}}{u_{k}^{+}+\varepsilon}\,\varphi$, $\varepsilon>0$,
we obtain
$$
\int\limits_{B_{\rho}(x_{0})}
\mathbf{A}(x, \nabla u_{k}^{+})\,\dfrac{u_{k}^{+}}{u_{k}^{+}+\varepsilon}\,\nabla\varphi\,dx
=-\varepsilon\int\limits_{B_{\rho}(x_{0})}
\frac{\mathbf{A}(x, \nabla u_{k}^{+})\nabla u_{k}^{+}}
{(u_{k}^{+}+\varepsilon)^{2}}\,\varphi\,dx\leqslant0,
$$
letting $\varepsilon\rightarrow0$, we arrive at the required statement.
\end{proof}

\subsubsection{Proof of Theorem \ref{th1.2}}

Let $x_{0}\in \partial\Omega$ be arbitrary and fix $R>0$ so that condition
{\rm (${\rm g}_{2}$)} be fulfilled in $B_{R}(x_{0})$.
Let $\rho\in (0,R)$ be arbitrary. Further we will assume that
\begin{equation}\label{eq4.1}
\sum_{i\in \mathbb{N}} \tau(r_{i}, \lambda(r_{i}))=+\infty,
\quad r_{i}:=\rho/2^{i},
\end{equation}
$$
\tau(r, \lambda(r))=r\,g_{x_{0}}^{-1}
\left( \frac{C(B_{r/4}(x_{0})\setminus\Omega,\, B_{r}(x_{0}); \lambda(r))}{r^{n-1}} \right).
$$
Under the condition \eqref{eq4.1} we need to prove that
$\lim\limits_{\Omega\ni x\rightarrow x_{0}}u(x)=f(x_{0})$.
This equality will be established if we show that
$\limsup\limits_{\Omega\ni x\rightarrow x_{0}}u(x)\leqslant f(x_{0})$
and $\liminf\limits_{\Omega\ni x\rightarrow x_{0}}u(x)\geqslant f(x_{0})$.
The proof of the both inequalities is completely similar and we will prove only the first one.
To prove the inequality $\limsup\limits_{\Omega\ni x\rightarrow x_{0}}u(x)\leqslant f(x_{0})$
we argue by contradiction and assume the
$\limsup\limits_{\Omega\ni x\rightarrow x_{0}}u(x)> f(x_{0})$.

Let $k$ be an arbitrary number such that
$f(x_{0})<k<\limsup\limits_{\Omega\ni x\rightarrow x_{0}}u(x)$, choose a number
$R_{0}=R_{0}(k)$ so that $k\geqslant\sup\limits_{\partial\Omega\cap B_{R_{0}}(x_{0})}f$
and assume without loss that $\rho\leqslant R_{0}$. In the ball $B_{\rho}(x_{0})$
consider the function $u^{+}_{k}$, which was defined in Lemma \ref{lem4.2pr1}. For
$r\in(0,\rho]$ we set $M_{k}(r):=\sup\limits_{\partial B_{r}(x_{0})} u_{k}$. By our choices
\begin{equation}\label{eq4.2pr1}
\lim\limits_{r\rightarrow0}M_{k}(r)=a>0.
\end{equation}
Fix a sufficiently large positive number $C_{4}$, which will be specified later
and choose $\widetilde{r}_{0}\in(0,\rho)$ from the condition
\begin{equation}\label{eq4.3pr1}
\tau(\widetilde{r}_{0},\lambda(\widetilde{r}_{0}))
\geqslant C_{4}\widetilde{r}_{0}.
\end{equation}
Condition \eqref{eq4.3pr1} can always be realized, since otherwise we would have
for all $r\in(0,\rho)$ that $\tau(r,\lambda(r))\leqslant C_{4}r$, and consequently
$$
\sum\limits_{i\in\mathbb{N}}
\tau(r_{i}, \lambda(r_{i}))\leqslant\gamma,
$$
reaching a contradiction to \eqref{eq4.1}.
By {\rm (${\rm g}_{1}$)} and \eqref{eq4.2pr1}, condition \eqref{eq4.3pr1} can be
rewritten as
\begin{equation}\label{eq4.4pr1}
C(B_{\widetilde{r}_{0}/4}(x_{0})\setminus\Omega, B_{\widetilde{r}_{0}}(x_{0});
M_{k}(\widetilde{r}_{0})\lambda(\widetilde{r}_{0}))\geqslant
\gamma(a)\,g(x_{0},C_{4})\geqslant \gamma(a)\,C_{4}^{\,p-1}.
\end{equation}

Let us construct an auxiliary solution $v=v(x,M_{k}(\widetilde{r}_{0})\lambda(\widetilde{r}_{0}))$
of the problem \eqref{eq1.14}, \eqref{eq1.15} in
$\mathcal{D}_{0}=B_{\widetilde{r}_{0}}(x_{0})\setminus E_{0}$,
$E_{0}=B_{\widetilde{r}_{0}/4}(x_{0})\setminus\Omega$.
If $\gamma(a)\,C_{4}^{\,p-1}\geqslant C_{3}$, then condition \eqref{eq1.14pr1} is violated and
by Theorem \ref{th1.4} we have with some $\eta\in(0,1)$,
$$
v(x)\geqslant \eta\,\tau(\widetilde{r}_{0},M_{k}(\widetilde{r}_{0})\lambda(\widetilde{r}_{0}))
\quad \text{for a.a.}
\ \ x\in B_{\widetilde{r}_{0}/2}(x_{0})\setminus B_{3\widetilde{r}_{0}/8}(x_{0}).
$$

Consider the function
$\widetilde{u}_{k,0}=M_{k}(\widetilde{r}_{0})-u_{k}^{+}$,
evidently $\widetilde{u}_{k,0}$ is a weak super-solution to Eq. \eqref{eq1.1} in
$B_{\widetilde{r}_{0}}(x_{0})$. Moreover,
$\widetilde{u}_{k,0}\geqslant v$ on $\partial\mathcal{D}_{0}$, so by Lemma \ref{lem4.1pr1}
$\widetilde{u}_{k,0}\geqslant v$ in $\mathcal{D}_{0}$, and from the previous we obtain
\begin{equation}\label{eq4.5pr1}
M_{k}(\widetilde{r}_{0})-M_{k}(\widetilde{r}_{0}/2)\geqslant
\eta\,\tau(\widetilde{r}_{0},M_{k}(\widetilde{r}_{0})\lambda(\widetilde{r}_{0})).
\end{equation}
If
$\tau(\widetilde{r}_{0}/2, \lambda(\widetilde{r}_{0}/2))
\geqslant C_{4}r_{0}/2$ we set $\widetilde{r}_{1}=\widetilde{r}_{0}/2$.
If
$\tau(\widetilde{r}_{0}/2, \lambda(\widetilde{r}_{0}/2))
\leqslant C_{4}r_{0}/2$, similarly to \eqref{eq4.3pr1} we find a number
$i\geqslant2$ such that
$\tau(\widetilde{r}_{0}/2^{i}, \lambda(\widetilde{r}_{0}/2^{i}))
\geqslant C_{4} r_{0}/2^{i}$.
Let $i_{1}\geqslant2$ be the maximal number satisfying the above condition,
in this case we set $\widetilde{r}_{1}=\widetilde{r}_{0}/2^{i_{1}}$, then
inequality \eqref{eq4.5pr1} can be rewritten in the form
$$
M_{k}(\widetilde{r}_{0})-M_{k}(\widetilde{r}_{1})\geqslant
\eta\,\tau(\widetilde{r}_{0},M_{k}(\widetilde{r}_{0})\lambda(\widetilde{r}_{0})).
$$

Further we define the sequence $\widetilde{r}_{j}$ by induction.
Suppose we have chosen $\widetilde{r}_{0}$, $\widetilde{r}_{1}$,
\ldots, $\widetilde{r}_{j}$ such that
\begin{equation}\label{eq4.6pr1}
\tau(\widetilde{r}_{i}, \lambda(\widetilde{r}_{i}))
\geqslant C_{4}\widetilde{r}_{i}, \quad
i=0,1, \ldots, j,
\end{equation}
\begin{equation}\label{eq4.7pr1}
M_{k}(\widetilde{r}_{i+1})\leqslant M_{k}(\widetilde{r}_{i})-
\eta\,\tau(\widetilde{r}_{i},M_{k}(\widetilde{r}_{i})\lambda(\widetilde{r}_{i})),
\quad i=0,1, \ldots, j-1.
\end{equation}
Let us show how to choose $\widetilde{r}_{j+1}$.
By \eqref{eq4.6pr1} similarly to \eqref{eq4.4pr1} we obtain
\begin{equation}\label{eq4.8pr1}
\tau(\widetilde{r}_{j},M_{k}(\widetilde{r}_{j})\lambda(\widetilde{r}_{j}))
\geqslant \gamma(a)g(x_{0},C_{4})
\geqslant \gamma(a)\,C_{4}^{\,p-1}.
\end{equation}
Let us construct an auxiliary solution $v=v(x,M_{k}(\widetilde{r}_{j})\lambda(\widetilde{r}_{j}))$
of the problem \eqref{eq1.14}, \eqref{eq1.15} in
$\mathcal{D}_{j}=B_{\widetilde{r}_{j}}(x_{0})\setminus E_{j}$,
$E_{j}=B_{\widetilde{r}_{j}/4}(x_{0})\setminus\Omega$.
Consider the function
$\widetilde{u}_{k,j}=M_{k}(\widetilde{r}_{j})-u_{k}^{+}$,
since $\widetilde{u}_{k,j}\geqslant v$ on $\partial\mathcal{D}_{j}$,
by Lemma \ref{lem4.1pr1} we obtain $\widetilde{u}_{k,j}\geqslant v$ in $\mathcal{D}_{j}$.
If $\gamma(a)\,C_{4}^{\,p-1}\geqslant C_{3}$, then condition \eqref{eq1.14pr1} is violated and by
Theorem \ref{th1.4} we have
\begin{equation}\label{eq4.9pr1}
M_{k}(\widetilde{r}_{j})-M_{k}(\widetilde{r}_{j}/2)\geqslant
\eta\,\tau(\widetilde{r}_{j},M_{k}(\widetilde{r}_{j})\lambda(\widetilde{r}_{j})).
\end{equation}
If
$\tau(\widetilde{r}_{j}/2, \lambda(\widetilde{r}_{j}/2))
\geqslant C_{4}\widetilde{r}_{j}/2$ we set $\widetilde{r}_{j+1}=\widetilde{r}_{j}/2$.
And if
$\tau(\widetilde{r}_{j}/2, \lambda(\widetilde{r}_{j}/2))\leqslant C_{4}r_{j}/2$,
similarly to \eqref{eq4.3pr1} we find a number $i\geqslant2$ such that
$\tau(\widetilde{r}_{j}/2^{i}, \lambda(\widetilde{r}_{j}/2^{i}))
\geqslant C_{4} \widetilde{r}_{j}/2^{i}$, let $i_{j+1}\geqslant2$ be the maximal number
satisfying the above condition, in this case we set
$\widetilde{r}_{j+1}=\widetilde{r}_{j}/2^{i_{j+1}}$.
Then inequality \eqref{eq4.9pr1} can be rewritten as
\begin{equation}\label{eq4.10pr1}
M_{k}(\widetilde{r}_{j})-M_{k}(\widetilde{r}_{j+1})\geqslant
\eta\,\tau(\widetilde{r}_{j},M_{k}(\widetilde{r}_{j})\lambda(\widetilde{r}_{j})).
\end{equation}

We have
\begin{equation}\label{eq4.11pr1}
\sum_{i\in \mathbb{N}} \tau(\widetilde{r}_{i}, \lambda(\widetilde{r}_{i}))=+\infty.
\end{equation}
Indeed, by our choices
$$
\begin{aligned}
\sum_{i\in \mathbb{N}}\tau(r_{i}, \lambda(r_{i}))&=
\sum_{\substack{i\in \mathbb{N}\\ r_{i}\neq \widetilde{r}_{i}}} \tau(r_{i}, \lambda(r_{i}))+
\sum_{i\in \mathbb{N}} \tau(\widetilde{r}_{i}, \lambda(\widetilde{r}_{i}))
\\
&\leqslant
\gamma\sum_{i\in \mathbb{N}}
\frac{\rho}{2^{i}}+
\sum_{i\in \mathbb{N}} \tau(\widetilde{r}_{i}, \lambda(\widetilde{r}_{i}))
\leqslant
\gamma\rho+
\sum_{i\in \mathbb{N}} \tau(\widetilde{r}_{i}, \lambda(\widetilde{r}_{i})),
\end{aligned}
$$
from which the required \eqref{eq4.11pr1} follows.

To complete the proof of Theorem \ref{th1.2} we sum up inequality \eqref{eq4.10pr1}
for $j=0, 1, 2, \ldots, l$.
By \eqref{eq4.2pr1} and {\rm (${\rm g}_{1}$)} we obtain that
$$
\eta\,\gamma(a)\sum\limits_{j=0}^{l}
\tau(\widetilde{r}_{j}, \lambda(\widetilde{r}_{j}))
\leqslant \sum\limits_{j=0}^{l}
(M_{k}(\widetilde{r}_{j})-M_{k}(\widetilde{r}_{j+1}))
\leqslant M_{k}(r_{0})\leqslant 2M.
$$
This implies that
$\sum\limits_{i\in \mathbb{N}} \tau(\widetilde{r}_{i},
\lambda(\widetilde{r}_{i}))\leqslant\gamma(a,M)<+\infty$,
which contradicts \eqref{eq4.11pr1}.
This completes the proof of Theorem \ref{th1.2}.



\subsection{Harnack's inequality for double-phase elliptic equations,
proof of Theorem \ref{th1.3}}

To prove Theorem \ref{th1.3} we need the following version of De\,Giorgi type lemma,
which was proved in \cite{SkrVoitNA20}.

Let $u$ be a bounded weak solution to Eq. \eqref{eq1.1}. Let $\overline{x}\in \Omega$ be
an arbitrary point. Consider a ball $B_{8r}(\overline{x})\subset\Omega$ and denote by
$\mu_{\pm}$ and $\omega$ non-negative numbers such that
$$
\mu_{+}\geqslant \esssup_{B_{r}(\overline{x})}u, \quad
\mu_{-}\leqslant \essinf_{B_{r}(\overline{x})}u, \quad
\omega=\mu_{+}-\mu_{-}.
$$
\begin{lemma}
Let conditions  {\rm (${\rm g}_{1}$)} and
{\rm (${\rm g}_{3}$)} be fulfilled. Fix $\xi$, $a\in (0,1)$, then there exists
number $\nu\in(0,1)$ depending only on the data and $a$ such that if
\begin{equation}\label{eq4.14}
\left| \left\{ x\in B_{r}(\overline{x}): u(x)\geqslant \mu_{+}-\xi\,\omega  \right\}  \right|
\leqslant \nu\, \theta^{-2n}(r)\, |B_{r}(\overline{x})|,
\end{equation}
then either
\begin{equation}\label{eq4.15}
\xi\,\omega \leqslant r
\end{equation}
or
\begin{equation}\label{eq4.16}
u(x)\leqslant \mu_{+}-a\,\xi\,\omega \quad
\text{for a.a.} \ x\in B_{r/2}(\overline{x}).
\end{equation}
Likewise, if
\begin{equation}\label{eq4.17}
\left| \left\{ x\in B_{r}(\overline{x}): u(x)\leqslant \mu_{-}+\xi\,\omega  \right\}  \right|
\leqslant \nu\, [\theta(r)]^{-2n}\, |B_{r}(\overline{x})|,
\end{equation}
then either \eqref{eq4.15} holds, or
\begin{equation}\label{eq4.18}
u(x)\geqslant \mu_{-}+a\,\xi\,\omega \quad
\text{for a.a.} \ x\in B_{r/2}(\overline{x}).
\end{equation}
\end{lemma}

Fix $x_{0}\in \Omega$ and for $0<8\rho<R$ construct the ball
$B_{\rho\tau}(x_{0})$, $\tau\in(0,1)$ and set $u_{0}:=u(x_{0})$.
Following Krylov and Safonov \cite{KrlvSfnv1980} consider the equation
$$
\max\limits_{B_{\rho\tau}(x_{0})}u=\frac{u_{0}}{2(1-\tau)^{l_{1}}}
\left( \frac{ \theta\big((1-\tau)\rho\big) }{\theta(\rho)} \right)^{l_{2}},
$$
where $l_{1}$, $l_{2}>0$ will be chosen later depending only on the known data.

Further we will assume that
\begin{equation}\label{eq4.17pr1}
u_{0}\geqslant C_{2}\rho\, \frac{\theta^{\frac{2n}{p-1}}(\rho)}{\lambda(\rho)}.
\end{equation}
Let $\tau_{0}\in (0,1)$ be the maximal root of the above equation.
Fix $\overline{x}$ by the condition
$$
u(\overline{x})=\max\limits_{B_{\rho\tau_{0}}(x_{0})}u=
\frac{u_{0}}{2(1-\tau_{0})^{l_{1}}}
\left( \frac{ \theta\big((1-\tau_{0})\rho\big) }{\theta(\rho)} \right)^{l_{2}}.
$$
Since $B_{\frac{\rho(1-\tau_{0})}{2}}(\overline{x})\subset
B_{\frac{\rho(1+\tau_{0})}{2}}(x_{0})$, by our choice of $\theta(\rho)$
we have
$$
\max_{B_{\frac{\rho(1-\tau_{0})}{2}}(\overline{x})}u
\leqslant \frac{2^{\,l_{1}-1}u_{0}}{(1-\tau_{0})^{l_{1}}}
\left(\frac{ \theta\big(\frac{1-\tau_{0}}{2}\rho\big)}{\theta(\rho)} \right)^{l_{2}}
=2^{\,l_{1}}u(\overline{x})
\left( \frac{\theta\big(\frac{1-\tau_{0}}{2}\rho\big)}
{\theta\big((1-\tau_{0})\rho\big) } \right)^{l_{2}}
\leqslant 2^{\,l_{1}+l_{2}} u(\overline{x}).
$$

\textsl{Claim.}
There exists a positive number $\nu\in(0,1)$ depending only on the known data
such that
\begin{equation}\label{eq4.18pr1}
\left| \left\{ x\in B_{\frac{\rho(1-\tau_{0})}{2}}(\overline{x}):
u(x)\geqslant \frac{u(\overline{x})}{2}  \right\}  \right|\geqslant \nu\,
\theta^{-2n}  \bigg( \frac{1-\tau_{0}}{2}\rho \bigg)
| B_{\frac{\rho(1-\tau_{0})}{2}}(\overline{x})|.
\end{equation}

Indeed, in the opposite case we apply \eqref{eq4.14}--\eqref{eq4.16} with the choices
$$
\mu_{+}=2^{\,l_{1}+l_{2}}u(\overline{x}), \quad
\xi\,\omega=\left(2^{\,l_{1}+l_{2}}-\frac{1}{2}\right)u(\overline{x}), \quad
a=\frac{2^{\,l_{1}+l_{2}}-\frac{3}{4}}{2^{\,l_{1}+l_{2}}-\frac{1}{2}}
\in(0,1).
$$
The condition \eqref{eq4.17pr1} obviously implies
$\xi\,\omega\geqslant (1-\tau_{0})\rho/2$. Therefore we can conclude that
$$
u(\overline{x})\leqslant\max\limits_{B_{\frac{\rho(1-\tau_{0})}{4}}(\overline{x})}u
\leqslant\frac{3}{4}\,u(\overline{x}),
$$
reaching a contradiction, which proves the claim.

Set $r=\dfrac{1-\tau_{0}}{2}\rho$,
$E:= B_{r}(\overline{x})\cap \left\{u\geqslant \dfrac{u(\overline{x})\lambda(\rho)}{2}\right\}$,
then inequality \eqref{eq4.18pr1} translates into
\begin{equation}\label{eq4.19pr1}
|E|\geqslant \nu\, \theta^{-2n}(r)\,|B_{r}(\overline{x})|.
\end{equation}

We will apply Theorem \ref{th1.4}, for this we consider an auxiliary solution
$v(x):=v(x,m)$, $m=\frac{1}{2}u(\overline{x})\lambda(\rho)$ of the problem
\eqref{eq1.14}, \eqref{eq1.15} in $\mathcal{D}=B_{4\rho}(\overline{x})\setminus E$.

First, we need to check the inequality
\begin{equation}\label{eq4.20pr1}
C(E, B_{4\rho}(\overline{x});m)\geqslant C_{3}\,\rho^{n-1}.
\end{equation}
Set $\Phi(x, {\rm v}):=\int\limits_{0}^{{\rm v}} g(x,s)\,ds$,
by {\rm (${\rm g}_{1}$)} $\Phi(x, {\rm v})\asymp G(x, {\rm v})$, $x\in\Omega$, ${\rm v}>0$.
Using the Poincar\'{e} inequality and conditions {\rm (${\rm g}_{1}$)}, {\rm (${\rm g}_{2}$)},
for any $\varphi\in W_{0}(B_{4\rho}(\overline{x}))$ such that $\varphi\geqslant1$ on $E$,
we obtain
$$
\begin{aligned}
&\frac{1}{\rho}\int\limits_{B_{4\rho}(\overline{x})}
g(\overline{x},m\varphi/\rho)\,\varphi\,dx\leqslant
\frac{\gamma}{m}\int\limits_{B_{4\rho}(\overline{x})}
\Phi(\overline{x},m\varphi/\rho)\,dx
\\
&\leqslant \gamma\int\limits_{B_{4\rho}(\overline{x})}
g(\overline{x},m\varphi/\rho)\,|\nabla\varphi|\,dx
\leqslant \gamma\int\limits_{B_{4\rho}(\overline{x})}
g(x,m\varphi/\rho)\,|\nabla\varphi|\,dx
\\
&\leqslant \frac{1}{2\gamma\rho}\int\limits_{B_{4\rho}(\overline{x})}
g(x,m\varphi/\rho)\,\varphi\,dx+\gamma\int\limits_{B_{4\rho}(\overline{x})}
g(x,m |\nabla\varphi|)\,|\nabla\varphi|\,dx
\\
&\leqslant \frac{1}{2\rho}\int\limits_{B_{4\rho}(\overline{x})}
g(\overline{x}, m\varphi/\rho)\,\varphi\,dx+\gamma
\int\limits_{B_{4\rho}(\overline{x})}
g(x, m|\nabla\varphi|)\,|\nabla\varphi|\,dx,
\end{aligned}
$$
which implies that
$$
\begin{aligned}
&\int\limits_{B_{4\rho}(\overline{x})}
g(x,m|\nabla\varphi|)\,|\nabla\varphi|\,dx
\geqslant
\frac{\gamma^{-1}}{\rho}\int\limits_{B_{4\rho}(\overline{x})}
g(\overline{x}, m\varphi/\rho)\,\varphi\,dx
\\
& \geqslant
\gamma^{-1} g(\overline{x},m/\rho)\,\frac{|E|}{\rho}
\geqslant \gamma^{-1}\nu\,\theta^{-2n}(r)\, g(\overline{x},m/\rho)\,
\left(\frac{r}{\rho}\right)^{n}\rho^{n-1}.
\end{aligned}
$$
And hence
\begin{equation}\label{eq4.21pr1}
C(E,B_{4\rho}(\overline{x});m)\geqslant
\gamma^{-1}\nu\,\theta^{-2n}(r)\, g(\overline{x},m/\rho)\,
\left(\frac{r}{\rho}\right)^{n}\rho^{n-1}.
\end{equation}
By {\rm (${\rm g}_{1}$)}, \eqref{eq4.17pr1} and our choice of $u(\overline{x})$ we have
$$
\begin{aligned}
g(\overline{x},m/\rho)&=
g\bigg( \overline{x}, \left( \frac{\theta(r)}{\theta(\rho)} \right)^{l_{2}}
\left(\frac{\rho}{r}\right)^{l_{1}}\frac{u_{0}\lambda(\rho)}{2\rho} \bigg)
\\
&\geqslant g\bigg( \overline{x}, \frac{C_{2}}{2}\left(\frac{\rho}{r}\right)^{l_{1}}
\theta^{\,l_{2}}(r)\,  \theta^{\frac{2n}{p-1}-l_{2}}(\rho) \bigg)
\\
&\geqslant
\gamma^{-1}C_{2}^{\,p-1}\,\theta^{\,l_{2}(p-1)}(r)\left(\frac{\rho}{r}\right)^{l_{1}(p-1)}
g\left(\overline{x}, \theta^{\frac{2n}{p-1}-l_{2}}(\rho)\right).
\end{aligned}
$$
Choosing $l_{1}$ and $l_{2}$ from the conditions $l_{1}=\dfrac{n}{p-1}$,
$l_{2}=\dfrac{2n}{p-1}$, from the last inequality we obtain
$$
g(\overline{x}, m/\rho)\geqslant \gamma^{-1}C_{2}^{\,p-1}\,\theta^{2n}(r)
\left(\frac{\rho}{r}\right)^{n}g(\overline{x},1)\geqslant
\gamma^{-1}C_{2}^{\,p-1}\,\theta^{2n}(r)
\left(\frac{\rho}{r}\right)^{n}.
$$
Therefore, by \eqref{eq4.21pr1}, choosing $C_{2}$ so that
$\gamma^{-1}\nu\,C_{2}^{\,p-1}\geqslant C_{3}$, we arrive at the required \eqref{eq4.20pr1}.

Using \eqref{eq4.20pr1} and Theorem \ref{th1.4}, we conclude that
$$
v(x)\geqslant \gamma^{-1}\rho\,
g_{\overline{x}}^{-1}\left( \frac{C(E,B_{4\rho}(\overline{x});m)}{\rho^{\,n-1}} \right)
\quad \text{for a.a.} \ x\in B_{2\rho}(\overline{x})\setminus B_{\rho}(\overline{x}).
$$
Inequality \eqref{eq4.21pr1} and condition {\rm (${\rm g}_{1}$)} imply that
$$
g^{-1}_{\overline{x}}\left( \frac{C(E,B_{4\rho}(\overline{x});m)}{\rho^{n-1}} \right)
\geqslant \gamma^{-1}\nu^{-\frac{1}{p-1}}\,\theta^{-\frac{2n}{p-1}}(r)
\left( \frac{r}{\rho} \right)^{\frac{n}{p-1}}\frac{m}{\rho}.
$$
And using our choices of $m$ and $u(\overline{x})$, from the previous we obtain
\begin{equation}\label{eq4.22pr1}
v(x)\geqslant \gamma^{-1}\,\theta^{-\frac{2n}{p-1}}(\rho)\,\lambda(\rho)\,u_{0},
\quad x\in B_{2\rho}(\overline{x})\setminus B_{\rho}(\overline{x}).
\end{equation}

By our construction $u\geqslant v$ on $\partial\mathcal{D}$, and therefore,
by Lemma \ref{lem4.1pr1} and \eqref{eq4.22pr1} we have
$$
\inf\limits_{B_{\rho}(x_{0})}u\geqslant
\inf\limits_{B_{2\rho}(\overline{x})}u\geqslant
\inf\limits_{\partial B_{2\rho}(\overline{x})}u\geqslant
\inf\limits_{\partial B_{2\rho}(\overline{x})}v\geqslant
\gamma^{-1}\,\theta^{-\frac{2n}{p-1}}(\rho)\,\lambda(\rho)\,u_{0},
$$
which completes the proof of Theorem \ref{th1.3}.




\vskip3.5mm
{\bf Acknowledgements.} The research of the second author was supported by grants of Ministry of Education and Science of Ukraine
(project numbers are 0118U003138, 0119U100421).

\bigskip

CONTACT INFORMATION

\medskip

Oleksandr V.~Hadzhy\\
Institute of Applied Mathematics and Mechanics,
National Academy of Sciences of Ukraine, Gen. Batiouk Str. 19, 84116 Sloviansk, Ukraine\\
aleksanderhadzhy@gmail.com

\medskip
Igor I.~Skrypnik\\Institute of Applied Mathematics and Mechanics,
National Academy of Sciences of Ukraine, Gen. Batiouk Str. 19, 84116 Sloviansk, Ukraine\\
Vasyl' Stus Donetsk National University,
600-richcha Str. 21, 21021 Vinnytsia, Ukraine\\iskrypnik@iamm.donbass.com

\medskip
Mykhailo V.~Voitovych\\Institute of Applied Mathematics and Mechanics,
National Academy of Sciences of Ukraine, Gen. Batiouk Str. 19, 84116 Sloviansk, Ukraine\\voitovichmv76@gmail.com


\begin{thebibliography}{99}


\bibitem{Alhutov97}
Yu.\,A.~Alkhutov,
The Harnack inequality and the H\"{o}lder property of solutions of nonlinear elliptic equations
with a nonstandard growth condition (Russian),
Differ. Uravn. \textbf{33} (1997), no.~12, 1651--1660;
translation in Differential Equations \textbf{33} (1997), no.~12, 1653--1663 (1998).



\bibitem{AlhutovMathSb05}
Yu.\,A.~Alkhutov,
On the H\"{o}lder continuity of $p(x)$-harmonic functions,
Sb. Math. \textbf{196} (2005), no. 1-2, 147--171.




\bibitem{AlhutovKrash04}
Yu.\,A.~Alkhutov, O.\,V.~Krasheninnikova,
Continuity at boundary points of solutions of quasilinear
elliptic equations with a nonstandard growth condition,
Izv. Ross. Akad. Nauk Ser. Mat. \textbf{68} (2004), no.~6, 3--60 (in Russian).



\bibitem{AlhutovKrash08}
Yu.\,A.~Alkhutov, O.\,V.~Krasheninnikova,
On the continuity of solutions of elliptic equations with a variable order of nonlinearity
(Russian),
Tr. Mat. Inst. Steklova \textbf{261}, (2008), Differ. Uravn. i Din. Sist., 7--15;
translation in Proc. Steklov Inst. Math. \textbf{261} (2008), no.~1--10.


\bibitem{AlkhSurnApplAn19}
Yu.\,A.~Alkhutov, M.\,D.~Surnachev,
A Harnack inequality for a transmission problem with $p(x)$-Laplacian,
Appl. Anal. \textbf{98} (2019), no. 1-2, 332--344.



\bibitem{AlkhSurnJmathSci20}
Yu.\,A.~Alkhutov, M.\,D.~Surnachev,
Harnack's inequality for the $p(x)$-Laplacian with a two-phase exponent $p(x)$,
J. Math. Sci. (N.Y.) \textbf{244} (2020), no. 2, 116--147.





\bibitem{AlkhSurnAlgAn19}
Yu.\,A.~Alkhutov, M.\,D.~Surnachev,
Behavior at a boundary point of solutions of the Dirichlet problem for the $p(x)$-Laplacian
(Russian), Algebra i Analiz \textbf{31} (2019), no. 2, 88--117;
translation in St. Petersburg Math. J. \textbf{31} (2020), no. 2, 251--271.


\bibitem{AntDiazShm2002_monogr}
S.\,N.~Antontsev, J.\,I.~D\'{\i}az, S.~Shmarev,
Energy Methods for Free Boundary Problems. Applications to Nonlinear PDEs and Fluid Mechanics,
in: Progress in Nonlinear Differential Equations and their Applications, vol. 48,
Birkhauser Boston, Inc., Boston, MA, 2002.



\bibitem{BarColMing}
P.~Baroni, M.~Colombo, G.~Mingione,
Harnack inequalities for double phase functionals,
Nonlinear Anal. \textbf{121} (2015), 206--222.



\bibitem{BarColMingStPt16}
P.~Baroni, M.~Colombo, G.~Mingione,
Non-autonomous functionals, borderline cases and related
function classes,
St. Petersburg Math. J. \textbf{27} (2016), 347--379.



\bibitem{BarColMingCalc.Var.18}
P. Baroni, M. Colombo, G. Mingione,
Regularity for general functionals with double phase, Calc. Var.
Partial Differential Equations \textbf{57}, 62 (2018). 



\bibitem{BenHarHasKarp20}
A. Benyaiche, P. Harjulehto, P. H\"{a}st\"{o}, A. Karppinen,
The weak Harnack inequality for unbounded supersolutions
of equations with generalized Orlicz growth,
arXiv:2006.06276v1 [math.AP].


\bibitem{BurchSkrPotAn}
K.\,O. Buryachenko, I.\,I. Skrypnik,
Local continuity and Harnack’s inequality for double-phase parabolic equations,
Potential Analysis, accepted for publication.



\bibitem{ColMing218}
M.~Colombo, G.~Mingione,
Bounded minimisers of double phase variational integrals,
Arch. Rational Mech. Anal.  \textbf{218} (2015), no. 1, 219--273.



\bibitem{ColMing15}
M.~Colombo, G.~Mingione, Regularity for double phase variational problems,
Arch. Rational Mech. Anal.  \textbf{215} (2015), no. 2, 443--496.



\bibitem{ColMingJFnctAn16}
M. Colombo, G. Mingione,
Calderon-Zygmund estimates and non-uniformly elliptic operators,
J. Funct. Anal. \textbf{270} (2016), 1416--1478.



\bibitem{DienHarHastRuzVarEpn}
L. Diening, P. Harjulehto, P. H\"{a}st\"{o}, M.~R\r{u}\v{z}i\v{c}ka,
Lebesgue and Sobolev Spaces with Variable Exponents,
in: Lecture Notes in Mathematics, 2017,
Springer, Heidelberg, 2011, x+509 pp.



\bibitem{Fan1995}
X. Fan, A Class of De\,Giorgi Type and H\"{o}lder Continuity of Minimizers of
Variational with $m(x)$-Growth Condition. China: Lanzhou Univ., 1995.



\bibitem{FanZhao1999}
X. Fan, D. Zhao,
A class of De\,Giorgi type and H\"{o}lder continuity,
Nonlinear Anal. \textbf{36} (1999) 295--318.



\bibitem{HarHastOrlicz}
P. Harjulehto, P. H\"{a}st\"{o},
Orlicz Spaces and Generalized Orlicz Spaces, in: Lecture Notes in Mathematics,
vol. 2236, Springer, Cham, 2019, p. X+169 http://dx.doi.org/10.1007/978-3-030-15100-3



\bibitem{HarHasZAn19}
P. Harjulehto, P. H\"{a}st\"{o},
Boundary regularity under generalized growth conditions,
Z. Anal. Anwend. \textbf{38} (2019), no. 1, 73--96.



\bibitem{HarHastLee18}
P. Harjulehto, P. H\"{a}st\"{o}, M. Lee,
H\"{o}lder continuity of quasiminimizers and $\omega$-minimizers of functionals with generalized
Orlicz growth,  arXiv:1906.01866v2 [math.AP].


\bibitem{HarHastToiv17}
P. Harjulehto, P. H\"{a}st\"{o}, O. Toivanen,
H\"{o}lder regularity of quasiminimizers under generalized growth conditions,
Calc. Var. Partial Differential Equations \textbf{56} (2017), no. 2, Art. 22, 26 pp.



\bibitem{Krash2002}
O.\,V. Krasheninnikova,
On the continuity at a point of solutions of elliptic equations with a nonstandard growth
condition,
Proc. Steklov Inst. Math. (2002), no. 1(236), 193--200.



\bibitem{KrlvSfnv1980}
N.\,V. Krylov, M.\,V. Safonov,
A property of the solutions of parabolic equations with measurable coefficients,
Izv. Akad. Nauk SSSR Ser. Mat. \textbf{44} (1980), no. 1, 161--175 (in Russian).



\bibitem{LadUr}
O.\,A.~Ladyzhenskaya, N.\,N.~Ural'tseva,
Linear and quasilinear elliptic equations,
Nauka, Moscow, 1973.


\bibitem{Landis_uspehi1963}
E.\,M. Landis,
Some questions in the qualitative theory of second-order elliptic equations
(case of several independent variables),
Uspehi Mat. Nauk \textbf{18} (1963), no. 1 (109), 3--62 (in Russian).




\bibitem{Landis_mngrph71}
E.\,M. Landis, Second Order Equations of Elliptic and Parabolic Type,
in: Translations of Mathematical Monographs, vol. 171, American Math. Soc.,
Providence, RI, 1998.



\bibitem{Lieberman91}
G.\,M.~Lieberman,
The natural generalization of the natural conditions of
Ladyzhenskaya and Ural'tseva for elliptic equations,
Comm. Partial Differential Equations \textbf{16} (1991), no. 2-3, 311--361.


\bibitem{Marcellini1989}
P.~Marcellini, Regularity of minimizers of integrals of the calculus of variations with
non standard growth conditions, Arch. Rational Mech. Anal.  \textbf{105} (1989), no.~3, 267--284.



\bibitem{Marcellini1991}
P.~Marcellini, Regularity and existence of solutions of elliptic equations with $p,q$-growth conditions,
J. Differential Equations \textbf{90} (1991), no. 1, 1--30.


\bibitem{OkNA20}
J. Ok,
Regularity for double phase problems under additional integrability assumptions,
Nonlinear Anal. \textbf{194} (2020) 111408.



\bibitem{Ruzicka2000}
M.~R\r{u}\v{z}i\v{c}ka,
Electrorheological fluids: Modeling and Mathematical Theory, in:
Lecture Notes in Mathematics, vol.~1748, Springer-Verlag, Berlin, 2000.


\bibitem{ShSkrVoit20}
M.\,A. Shan,  I.\,I. Skrypnik,  M.\,V. Voitovych,
Harnack's inequality for quasilinear elliptic equations with
generalized Orlicz growth, arxiv:2008.03744v1 [math.AP].



\bibitem{IVSkr1983}
I.\,V. Skrypnik,
Pointwise estimates of certain capacitative potentials, in:
General theory of boundary value problems, pp. 198--206, Naukova Dumka, Kiev, 1983 (in Russian).


\bibitem{IVSkrMetodsAn1994}
I.\,V. Skrypnik,
Methods for Analysis of Nonlinear Elliptic Boundary Value Problems,
in: Translations of Mathematical Monographs, vol. 139, American Math. Soc.,
Providence, RI, 1994.




\bibitem{IVSkrSelWorks}
I.\,V. Skrypnik, Selected works, in:
Problems and Methods. Mathematics. Mechanics. Cybernetics, vol.~1,
Naukova Dumka, Kiev, 2008 (in Russian).



\bibitem{SkrVoitUMB19}
I.\,I. Skrypnik, M.\,V. Voitovych,
$\mathfrak{B}_{1}$ classes of De Giorgi, Ladyzhenskaya and Ural'tseva
and their application to elliptic and parabolic equations with nonstandard growth,
Ukr. Mat. Visn. \textbf{16} (2019), no. 3, 403--447.



\bibitem{SkrVoitNA20}
I.\,I. Skrypnik, M.\,V. Voitovych,
$\mathcal{B}_{1}$ classes of De Giorgi-Ladyzhenskaya-Ural'tseva and
their applications to elliptic and parabolic equations with generalized Orlicz
growth conditions, Nonlinear Anal. \textbf{202} (2021) 112135.



\bibitem{VoitNA19}
M.\,V. Voitovych,
Pointwise estimates of solutions to $2m$-order quasilinear elliptic
equations with $m$-$(p, q)$ growth via Wolff potentials,
Nonlinear Anal. \textbf{181} (2019) 147--179.



\bibitem{SurnPrepr2018}
M.\,D. Surnachev,
On Harnack's inequality for $p(x)$-Laplacian (Russian),
Keldysh Institute Preprints 10.20948/prepr-2018-69, \textbf{69} (2018), 1--32.



\bibitem{Weickert}
J.~Weickert,
Anisotropic Diffusion in Image Processing, in:
European Consortium for Mathematics in Industry, B.G. Teubner, Stuttgart, 1998.


\bibitem{ZhikIzv1983}
V.\,V.~Zhikov, Questions of convergence, duality and averaging for functionals
of the calculus of variations (Russian),
Izv. Akad. Nauk SSSR Ser. Mat. \textbf{47} (1983), no.~5, 961--998.


\bibitem{ZhikIzv1986}
V.\,V.~Zhikov,
Averaging of functionals of the calculus of variations and elasticity theory
(Russian),
Izv. Akad. Nauk SSSR Ser. Mat. \textbf{50}, (1986), no.~4, 675--710, 877.


\bibitem{ZhikJMathPh94}
V.\,V.~Zhikov,
On Lavrentiev's phenomenon,
Russian J. Math. Phys. \textbf{3} (1995), no.~2, 249--269.


\bibitem{ZhikJMathPh9798}
V.\,V.~Zhikov,
On some variational problems,
Russian J. Math. Phys. \textbf{5} (1997), no. 1, 105--116 (1998).


\bibitem{ZhikPOMI04}
V.\,V. Zhikov,
On the density of smooth functions in Sobolev-Orlicz spaces
(Russian), Zap. Nauchn. Sem. S.-Peterburg. Otdel. Mat. Inst. Steklov. (POMI) \textbf{310} (2004),
Kraev. Zadachi Mat. Fiz. i Smezh. Vopr. Teor. Funkts. 35 [34], 67--81, 226;
translation in J. Math. Sci. (N.Y.) \textbf{132} (2006), no. 3, 285--294.


\bibitem{ZhikKozlOlein94}
V.\,V.~Zhikov, S.\,M.~Kozlov, O.\,A.~Oleinik,
Homogenization of differential operators and integral functionals,
Springer-Verlag, Berlin, 1994.


\bibitem{ZhikPast2008MatSb}
V.\,V.~Zhikov, S.\,E.~Pastukhova,
On the improved integrability of the gradient of solutions of elliptic equations with a variable nonlinearity exponent (Russian), Mat. Sb. \textbf{199} (2008), no.~12, 19--52;
translation in Sb. Math. \textbf{199} (2008), no. 11--12, 1751--1782.

\end{thebibliography}
\end{document}